\documentclass[11pt]{article}

\newcommand{\insieme}[1]{\left\{ #1 \right\}}
\usepackage{fullpage}

\usepackage{enumerate}
\usepackage{amsmath,amsfonts,amssymb,amsthm, bbold}
\usepackage{svg}
\usepackage{graphics,esint}
\usepackage{epsfig}
\usepackage{xcolor}
\usepackage{xspace}
\definecolor{pingreen}{rgb}{0,39,14}

\usepackage{xfrac}
\usepackage{comment}
\usepackage{hyperref}
\usepackage{cleveref}

\crefname{section}{§}{§§}
\Crefname{section}{§}{§§}

\def\vol{\mathrm{vol}}

\def\Int{\mathrm{Int}}

\def\Acal{\mathcal{A}}

\def\Hcal{\mathcal{H}}

\def\lhs{l.h.s.\xspace}

\newcommand{\hausd}{\mathcal H} 

\makeatletter
\newtheorem*{rep@theorem}{\rep@title}
\newcommand{\newreptheorem}[2]{%
\newenvironment{rep#1}[1]{%
 \def\rep@title{#2 \ref{##1}}%
 \begin{rep@theorem}}%
 {\end{rep@theorem}}}
\makeatother

\newtheorem{theorem}{Theorem}[section]

\newtheorem{lemma}[theorem]{Lemma}
\newtheorem{definition}[theorem]{Definition}

\newtheorem{remark}[theorem]{Remark}

\newtheorem{proposition}[theorem]{Proposition}

\newreptheorem{theorem}{Theorem}

\newcommand{\R}{\mathbb{R}}
\newcommand{\Z}{\mathbb{Z}}

\newcommand{\Fcal}{\mathcal{F}}

\def\SS{\mathcal S}

\def\N{\mathbb N}
\def\eps{\varepsilon}

\def\per{\mathrm{Per}}

\def\eps{\varepsilon}

\def\d {\,\mathrm {d}}
\def\dx{\,\mathrm {d}x}
\def\dz{\,\mathrm {d}z}
\def\ds{\,\mathrm {d}s}
\def\du{\,\mathrm {d}u}
\def\dv{\,\mathrm {d}v}
\def\dt{\,\mathrm {d}t}
\def\dy{\,\mathrm {d}y}

\numberwithin{equation}{section}

\usepackage{authblk}

\author[1]{Sara Daneri\thanks{sara.daneri@gssi.it}}
\author[2]{Eris Runa\thanks{eris.runa@gmail.com}}
\affil[1]{Gran Sasso Science Institute, L'Aquila, Italy}
\affil[2]{Deutsche Bank AG, Berlin, Germany}

  \title{ Pattern formation for a local/nonlocal interaction functional arising in colloidal systems}

  \date{}
  
\begin{document}

  	\maketitle

  	\begin{abstract}
  		In this paper we study pattern formation for a physical local/nonlocal interaction functional where the local attractive  term is given by the $1$-perimeter and the nonlocal repulsive term is the Yukawa (or screened Coulomb) potential. 
  		This model is physically interesting as it is the $\Gamma$-limit of a double Yukawa model used to explain and simulate pattern formation in colloidal systems \cite{BBCH,CCA,IR,GCLW}. 
  		Following a strategy introduced in~\cite{DR} we prove that in a suitable regime minimizers are periodic stripes, in any space dimension. 
  	\end{abstract}
  

  	\section{Introduction}
  	
   The ability of matter to arrange itself in periodic structures is often referred to as spontaneous pattern formation. 
   This phenomenon is of fundamental importance in Science, Technology, Engineering and Mathematics and it is often caused by  the interaction between local attractive and nonlocal repulsive forces.

  	Instances of spontaneous pattern formation at a mesoscopic scale are that showed by certain suspensions of charged colloids, polymers and also by protein solutions, when the attractive and repulsive forces compete at some strength ratio. The phenomenon is since at least two decades object of experimental and computational investigation (see among the various papers on the subject \cite{poon, nat, BBCH, Bores, CCA,GCLW, IR}). In particular, one can observe gathering of the particles in lamellas (stripes) or bubbles (clusters) according to the different regimes between the two mutual interactions. These self-assembly processes play a crucial role in applications such as the production of photonic crystals, the possibility to control the formation of clusters in various diseases, nanolithography or gelation processes.
  	
  	For colloidal systems, the long-range repulsive forces have been shown on theoretical grounds to be represented by a Yukawa (or \emph{screened Coulomb}) potential \cite{DL,VO} (the so-called DLVO Theory).
  	
  	The kernel of the Yukawa potential is the following
  	\begin{equation}
  		\label{eq:Yu}
  		K_\mu(\zeta):=\frac{e^{-\mu|\zeta|}}{|\zeta|^{d-2}}, \quad\mu>0
  	\end{equation} 
  	if $d\geq 3$ and $K_\mu(\zeta):=-e^{-\mu|\zeta|}\ln(|\zeta|)$ if $d=2$.  
  	
  	The Yukawa potential was introduced in the 30s by Yukawa in particle physics \cite{Yu}. Other than for electrolytes and colloids, it is used also in plasma physics,  where it represents the potential of a charged particle in a weakly
  	nonideal plasma, in solid state physics, where it describes the effects of a charged particle in a sea of conduction
  	electrons, and in quantum mechanics.

  	Pattern formation in models for colloid particles involving the Yukawa potential as  repulsive term has been numerically studied in several papers (see e.g. \cite{ BBCH, Bores,CCA, GCLW, IR}) and lamellar (striped) phases has been  observed in suitable regimes. 
  	
  	In particular, some of these models (see e.g. \cite{BBCH,CCA,GCLW,IR}) use as short range attractive term the Yukawa potential with opposite sign and  parameter $\mu$ much larger than the one appearing in the repulsive Yukawa. A model of this kind is the following: for $d\geq1$, $\beta>1$, $L>0$, $E\subset\R^d$ $[0,L)^d$-periodic and $J>0$, consider
  	
  	\begin{align}
  	\tilde{\mathcal E}_{\beta,J,L}(E):=&\frac{1}{L^d}\Big(JC_{\beta,L}\int_{[0,L)^d }\int_{ \R^d}|\chi_E(x+\zeta)-\chi_E(x)|K_\beta(\zeta)\d\zeta\dx\notag\\
  	&-\int_{[0,L)^d }\int_{ \R^d}|\chi_E(x+\zeta)-\chi_E(x)|K_1(\zeta)\d\zeta\dx\Big),\label{E:Eintro}
  	\end{align}
  	where $C_{\beta,L}$ is a positive normalization constant depending on $\beta$ and $L$.

  	In Section \ref{sec:gammaconv} we show that, for a natural choice of the constant $C_{\beta,L}$ (see \eqref{eq:cbeta}) and substituting the Euclidean norm $|\cdot|$ with the $1$-norm $|z|_1=\sum_{i=1}^d|z_i|$ in \eqref{eq:Yu}, the functionals $	\tilde{\mathcal E}_{\beta,J,L}$ $\Gamma$-converge as $\beta\to+\infty$ to the following functional

  	\begin{equation}\label{E:F}
  		\begin{split}
  			\tilde{\Fcal}_{J,L}(E) := \frac{1}{L^d}\Big(J \per_1(E, [0,L)^d)  -  \int_{[0,L)^d}\int_{\R^d} |\chi_{E}(x + \zeta) - \chi_{E}(x) | K_1(\zeta) \d\zeta\dx \Big), 
  		\end{split}
  	\end{equation}
  	
  	where 
  	\begin{equation*} 
  		\begin{split}
  			\per_{1}(E,[0,L)^d):=\int_{\partial E\cap [0,L)^d}|\nu^E(x)|_1\, d\mathcal H^{d-1}(x),\quad \text{$|z|_1=\sum_{i=1}^d|z_i|$},
  		\end{split}
  	\end{equation*} 
  	with $\nu^E(x)$ exterior normal to $E$ in $x$, is the $1$-perimeter of $E$, and $K_1$ is from now on the  Yukawa-type potential
  	\begin{equation}\label{eq:Yu1}
  		K_1(\zeta):=\frac{e^{-|\zeta|_1}}{|\zeta|_1^{d-2}},
  	\end{equation}
  	
  	namely the potential obtained substituting the Euclidean norm in \eqref{eq:Yu} with the $1$-norm. Notice that, since we  assume periodicity of the sets w.r.t. $[0,L)^d$, the choice of the norm does not reduce the underlying symmetries of the problem, namely those w.r.t. permutation of coordinates.
  	
  	A functional of this kind (namely with the perimeter as attractive term and the Yukawa as repulsive term), but in a different regime, appears also in a series of papers \cite{Mur,GolSer1,GolSer2} in connection with the sharp interface limit of the Ohta-Kawasaki model for small volume fractions.

  	The potential \eqref{eq:Yu1} behaves, for short range interactions, like the Coulomb potential and, for long range interactions, like a strong decaying potential, analogously to the generalized anti-ferromagnetic potentials considered in \cite{DR,GR,GiuSeirGS}.  
  	
  	Our aim in this paper is to characterize minimizers of \eqref{E:F} for a suitable range of $J$.
  	
  	Evolutionary problems for attractive-repulsive models have been studied by various authors (see e.g. \cite{BDKMS,CDFL,CDKS,BCPS}), and in \cite{CDKS,BCPS} exponentially decaying kernels have been considered.
  	
  	The main difficulty in showing pattern formation lies in the fact that the functional exhibits a larger group of symmetries than the expected minimizers.

  	For one-dimensional models, where the symmetry breaking does not occur, pattern formation has been proved 
  	among others in \cite{muller1993singular,glllRP3}, 
  	using either convexity methods or reflection positivity techniques.
  	
  	In more space dimensions, in the discrete setting, pattern formation was shown in \cite{GiuSeirGS} for the kernel $\tilde K(\zeta)=\frac{1}{(|\zeta|_1+1)^p}$ and $p>2d$.
  	
  	In the continuous setting, pattern formation was shown for the first time for a functional which is invariant under permutation of coordinates in \cite{DR}. There the kernel $\tilde K$ can have any  exponent $p\geq d+2$.
  	
   Both in the discrete and  continuous case, one can show that there is a critical constant $J_c$ such that for every $J > J_c$ the minimizers are trivial, namely the minimizers are either the empty set or the whole $\R^d$. In both \cite{GiuSeirGS,DR}, the authors show that for $J$ close enough to $J_c$ the minimizers are a union of periodic stripes. For a comparision of the two papers see~\cite{DR}. From now on we will concentrate with the continous case. 

 In order to be more precise, let us recall that in the continuous setting a union of stripes is a $[0,L)^d$-periodic set which is, up to Lebesgue null sets, of the form $V_i^\perp+\hat Ee_i$ for some $i\in\{1,\dots,d\}$, where $V_i^\perp$ is the $(d-1)$-dimensional subspace orthogonal to $e_i$ and  $\hat{E}\subset \R$ with $\hat E\cap [0,L)=\cup_{k=1}^N(s_i,t_i)$. 
  	A union of stripes is periodic if $\exists\, h>0$, $\nu\in\R$ s.t. $\hat E\cap [0,L)=\cup_{k=0}^N(2kh+\nu,(2k+1)h+\nu)$.
  	
  	In this paper we  prove an analogous result for the functional \eqref{E:F}. 
  	
  	While for the power-like potential $\tilde K$ the physical exponents $p=d+1$ (thin magnetic films), $p=d$ (3D-micromagnetics) and $p=d-2$ (Coulomb potential) remain excluded by the results in \cite{GiuSeirGS,DR}, here we are able to prove pattern formation for a physical model.
  	
  	Let us now state our results precisely. First of all, there exists also in this case a critical constant $\tilde J_{\infty}$ such that if $J>\tilde J_\infty$ then the only minimizers are $\R^d$ and the empty set. Such a constant is given by 
  	\[
  	\tilde J_\infty:=\int_{ \R^d} |\zeta_1|K_1(\zeta)\d\zeta.
  	\]
  	
  	What one expects is that for values of $J$ strictly below $\tilde J_\infty$ minimizers are periodic unions of stripes of optimal period.
  	
  	Therefore one sets 
  	\[
  	\tilde J_M=\int_{ \R^{d-1}}\int_{-M}^M|\zeta_1|K_1(\zeta)\d\zeta 
  	\]
  	and considers the functional  
  	\begin{equation}\label{E:M}
  		\tilde{ \Fcal}_{\tilde J_M,L}(E)=\frac{1}{L^d}\Big(\tilde J_M\per_1(E,[0,L)^d)-\int_{[0,L)^d}\int_{\R^d}|\chi_E(x+\zeta)-\chi_E(x)|K_1(\zeta)\d\zeta\dx\Big).
  	\end{equation}
  	One can see that minimizing \eqref{E:M} in the class of periodic unions of stripes, for $1\ll 2M\leq L$, those with optimal energy have width and distance of order  $h_{M}\leq M$, $\frac{h_{M}}{M}\to1$ as $M\to+\infty$ and energy of order $e^*_M\geq-e^{-\alpha_MM}$, with $\alpha_M\leq1$, $\alpha_M\to1$ as $M\to+\infty$.  
  	
  	Therefore it is natural to  rescale the spacial variables and the functional  so that the optimal width and distance for unions of stripes is $O(1)$ and  the energy is $O(1)$.
  	
  	Setting $M\zeta'=\zeta$, $Mx'=x$, and $\tilde{ \Fcal}_{\tilde J_M,L}(E)=-e^*_M\Fcal_{M,L/M}(E/M)$ one ends up considering the rescaled functional
  	\begin{equation}
  		\Fcal_{M,L}(E)=\frac{M^2}{L^d}\Big(J_M\per_1(E,[0,L)^d)-\int_{[0,L)^d}\int_{\R^d}|\chi_E(x+\zeta)-\chi_E(x)|\bar K_M(\zeta)\dx\d\zeta\Big)
  	\end{equation}
  	where
  	\begin{equation*}
  		\begin{split}
  			\bar K_M(\zeta)  = \frac{-1}{{e^*_M} |\zeta|^{d-2}} e^{-M|\zeta |}.
  		\end{split}
  	\end{equation*}
  	and
  	\[
  	J_M=\int_{ \R^{d-1}}\int_{-1}^1\bar K_M(\zeta)|\zeta_1|\d\zeta.
  	\]

  	For fixed $M > 0$, consider first for all $L > 0$ the minimal value obtained by $\Fcal_{M,L}$ on $[0,L)^d$-periodic unions of stripes
  	and then the minimal among these values as $L$ varies in $(0,+\infty)$. We will denote this value by $e^*_M$. 
  	
  	By the reflection positivity technique (see~Section~\ref{sec:yukawa_1D_problem}), this value is attained on periodic stripes of width and distance $h^*_M > 0$, which is unique provided $M$ is large enough (see Theorem \ref{thm:unique}).
  	
  	Our main theorem is the following:

  	\begin{theorem}
  		\label{thm:main_yukawa}
  		There exists a constant $M_0$ such that for every $M> M_0$ and $L = 2kh^*_M$ for some $k\in \N$, then the minimizer of ${\Fcal}_{M,L}$  are optimal stripes of width and distance $h^*_M$. 
  	\end{theorem}
  	
  	In Theorem \ref{thm:main_yukawa} notice that $M_0$ is independent of $L$.
  	
  	Notice that  the $[0,L)^d$-periodic boundary conditions were imposed in order to give sense to the functional which is otherwise not well-defined. If one is interested to show that optimal periodic stripes of width and distance $h^*_M$ are ``optimal'' if one varies also the periodicity, then it is not difficult  to see that Theorem~1.3 is sufficient. This process is similar to the ``thermodynamic limit'' which is of particular relevance in physics.

   By adapting some of the arguments in  \cite{DR}, one can provide a characterization of minimizers of $\Fcal_{M,L}$ also for arbitrary $L$, but this time with $M$ larger than a constant depending on $L$. Namely, one has the following  
  	
  	\begin{theorem}\label{T:main2}
  		Let $L>0$. Then there exists $\bar{M}>0$ such that $\forall\,M\geq \bar M$ there exists $h_{M,L}$  such that  the minimizers of $\mathcal F_{M,L}$ are periodic stripes of width and distance $h_{M,L}$.
  	\end{theorem} 
  	
  	According to the next theorem, when $L$ is large then $h_{M,L}$ is close to $h^*_M$.
  	
  	\begin{theorem}\label{T:main3}
  		There exists $C>0$ and $\hat M>0$ such that for every $M>\hat M$ the width and distance $h_{M,L}$ of minimizers of $\Fcal_{M,L}$ satisfies
  		\[
  		\Big|h_{M,L}-{h^*_M}\Big|\leq \frac{C}{L}.
  		\]
  	\end{theorem}

As discussed in \cite{BBCH,CCA,IR,GCLW}, one of the possible models used to show gelification in charged colloids and pattern formation is to consider both as attractive and as repulsive term the Yukawa potential, with different signs and appropriate rescaling (see \eqref{E:Eintro}). 

 



Therefore, the following $\Gamma$-convergence result connects our analytical results  with the ones obtained in the above cited experiments and simulations for the functional \eqref{E:Eintro}.

\begin{theorem}\label{thm:gammayukintro}
         The functionals $\tilde{\mathcal E}_{\beta,J,L}$ defined in \eqref{E:Eintro} $\Gamma$-converge in the $L^1$ topology as $\beta\to+\infty$ and up to subsequences to the functional $\tilde{\mathcal F}_{J,L}$ defined in \eqref{E:F}.
\end{theorem}

   The main theorems in this paper are Theorem~\ref{thm:main_yukawa} and  Theorem~\ref{thm:gammayukintro}.
   The proofs of Theorem~\ref{T:main2} and Theorem~\ref{T:main3} can be obtained by using the new estimates present in this paper and the strategy contained in~\cite{GR,DR}.
   For this reason in this paper we will focus on the proof of  Theorem~\ref{thm:main_yukawa} and Theorem~\ref{thm:gammayukintro}.

  	
  	The general strategy of the proof is similar to the one introduced in \cite{DR} and consists in the following steps: decomposition of the functional in terms which penalize deviations from being a stripe (\cite{GR,DR}); decomposition of $\R^d$ in different regions according to how much in a region the set $E$ ``resembles'' a stripe (\cite{GiuSeirGS,DR}); rigidity estimate to prove that in the limit $M\to+\infty$ minimizers approach a striped structure (\cite{GR,DR}); stability estimates to prove that once close to a stripe the most convenient thing is to be flat (\cite{DR}); use of the reflection positivity technique.
  	
  	However, the rigidity estimate (see Lemma \ref{lemma:yukawa_local_rigidity}) and the stability lemma (see Lemma \ref{lemma:yukawa_stimaContributoVariazionePiccola}) base on the specific  properties of the Yukawa kernel and are therefore different.
  	Finally, in Section \ref{sec:gammaconv} we prove the Theorem~\ref{thm:gammayukintro}.
  
      %

  \subsection{Structure of the paper}
  	This paper is organized as follows:   in Section \ref{sec:prel}, after setting the notation, we estimate the energy and width of optimal stripes and we rescale the functional accordingly. Then, we identify suitable quantities that penalize deviations from being a union of stripes. In Section \ref{sec:yukawa_1D_problem} we show that in the interesting regime the width of minimizers is uniquely determined and minimizers are periodic. Section \ref{sec:yukawa_taul} contains the main rigidity estimate and Section \ref{sec:main} the proof of Theorem \ref{thm:main_yukawa}. In Section \ref{sec:gammaconv} we prove the $\Gamma$-convergence of the double Yukawa functional \eqref{E:Eintro} to \eqref{E:F} as $\beta\to+\infty$.

 \section{Setting and preliminary results}
 \label{sec:prel}

 In this section, we set the notation and we introduce some preliminary results in the spirit of those given in \cite{GR} and \cite{DR}, which will be necessary to carry on our analysis. 
 
 \subsection{Notation and preliminary definitions}
 
 In the following, we let $\N=\{1,2,\dots\}$, $d\geq 1$. On $\R^d$, we let $\langle\cdot,\cdot\rangle$ be the Euclidean scalar product and $|\cdot|$ be the Euclidean norm.  We let $(e_1,\dots,e_d)$ be the canonical basis in $\R^d$ and  for $y\in\R^d$ we let $y_i=\langle y,e_i\rangle e_i$ and $y_i^\perp:=y-y_i$. For $y\in\R^d$, let 
 $|y|_1=\sum_{i=1}^d|y_i|$ be its $1$-norm  and $|y|_\infty=\max_i|y_i|$ its $\infty$-norm.

 Given a measurable set $A\subset\R^d$, let us denote by $\mathcal H^{d-1}(A)$ its $(d-1)$-dimensional Hausdorff measure and $|A|$ its Lebesgue measure. Moreover, let $\chi_A:\R^d\to\R$ the function defined by
 \begin{equation}
 \chi_A(x)=\left\{\begin{aligned}
 &1 && &\text{if $x\in A$}\\
 &0 && &\text{if $x\in\R^d\setminus A$}
 \end{aligned}\right.
 \end{equation}
 and by $\mathbb 1^\infty_A:\R^d\to\R\cup\{+\infty\}$ the function
  \begin{equation}
 \mathbb 1^\infty_A(x)=\left\{\begin{aligned}
 &+\infty && &\text{if $x\in A$}\\
 &0 && &\text{if $x\in\R^d\setminus A$}
 \end{aligned}\right.
 \end{equation}
 
 A set $E\subset\R^d$ is of (locally) finite perimeter if the distributional derivative of $\chi_E$ is a (locally) finite measure. We let $\partial E$ be the reduced boundary of $E$.  We call $\nu^E$ the exterior normal to $E$.
 
 The first term of our functional is, up to a constant, the $1$-perimeter of a set relative to $[0,L)^d$, namely
 
 \[
 \per_1(E,[0,L)^d):=\int_{\partial E\cap [0,L)^d}|\nu^E(x)|_1\d\mathcal H^{d-1}(x)
 \]
 and, for $i\in\{1,\dots,d\}$ 
 \begin{equation}
 \label{eq:yukawa_perI}
 \per_{1i}(E,[0,L)^d)=\int_{\partial E\cap [0,L)^d}|\nu^E_i(x)|\d\mathcal H^{d-1}(x),
 \end{equation}
 thus $\per_1(E,[0,L)^d)=\sum_{i=1}^d\per_{1i}(E,[0,L)^d)$. Notice that in the definition of $\per_1$ the norm applied to the exterior normal $\nu_E$ is the $1$-norm, instead of the Euclidean norm used to define the standard perimeter. 

 Because of  periodicity, w.l.o.g.  we always assume that $|D\chi_E|(\partial [0,L)^d)=0$, being $\chi_E$ the characteristic function of $E$ and $D\chi_E$ its distributional derivative.
 
 When $d=1$ one can define 
 \[
 \per_1(E, [0,L))  =\per(E,[0,L))= \#(\partial E \cap [0,L) ),
 \]
 where $\partial E$ is the reduced boundary of $E$.
 
 While writing slicing formulas, with a slight abuse of notation we will sometimes identify $x_i\in[0,L)^d$ with its coordinate in $\R$ w.r.t. $e_i$ and $\{x_i^\perp:\,x\in[0,L)^d\}$ with $[0,L)^{d-1}\subset\R^{d-1}$. 
 
 In Section \ref{sec:yukawa_taul} we will have to apply slicing on small cubes around a point. Therefore we need to introduce the following notation. For $z\in[0,L)^d$ and $r>0$, we define $Q_r(z)=\{x\in\R^d:\,|x-z|_\infty\leq r\}$.
 For $r> 0$ and $x^{\perp}_i$ we let $Q_{r}^{\perp}(x^\perp_{i}) = \{z^\perp_{i}:\, |x^{\perp}_{i} - z^{\perp}_{i} |_\infty \leq r  \}$ or we think of $x_i^\perp\in[0,L)^{d-1}$ and $Q_r^\perp(x_i^\perp)$ as a subset of $\R^{d-1}$. 
 Since the subscript $i$ will be always present in the centre (namely $x_i^\perp$) of such $(d-1)$-dimensional cube, the implicit dependence on $i$ of $Q_r^\perp(x_i^\perp)$ should be clear. 
 We denote also by $Q^i_r(t_i)\subset\R$ the interval of length $r$ centred in $t_i$.
 
 Now, let us turn to the elements defining the nonlocal  term in \eqref{E:F}. The kernel is the so called Yukawa kernel
 \begin{equation}
 K_1(\zeta)=\frac{e^{-|\zeta|_1}}{|\zeta|_1^{d-2}}, \quad\zeta\in\R^d,
 \end{equation}
 
 up to considering the $1$-norm in the exponential instead of the Euclidean norm. 
 
 As commented in the Introduction, such a kernel is both physical and reflection positive, namely it satisfies the following property (see property (2.6) in \cite{DR}): the function
 \begin{equation*}
 \begin{split}
 \widehat K_1 (t) := \int_{\R^{d-1}} K_1(t, \zeta_2,\ldots,\zeta_d)  \d\zeta_2\cdots\d\zeta_d. 
 \end{split}
 \end{equation*}
 
 is the Laplace transform of a nonnegative function (see Lemma \ref{lemma:yukawa_laplace}).
 
 Notice that 
 \begin{equation}
 \widehat K_1(t)=e^{-|t|}C(t),
 \end{equation}
 where $0<C(t)\leq \bar C$ and $C(t)\to0$ as $|t|\to+\infty$.

As in \cite{GiuSeirGS, GR, DR}, there exists a critical constant $\tilde J_{\infty}$ such that if $J>\tilde J_\infty$, the functional \eqref{E:F} is nonnegative and therefore has trivial minimizers. Such a constant is given by

\[
\tilde J_\infty:=\int_{ \R^d} |\zeta_1|K_1(\zeta)\d\zeta.
\] 

A proof of this fact is analogous to \cite{GR}[Proposition 3.5], so we omit it.

Letting, for $M\geq0$,

\[
\tilde J_M:=\int_{\R^{d-1}}\int_{-M}^M |\zeta_1|{K}_1(\zeta) \d\zeta=\int_{-M}^M|\zeta_1|\widehat K_1(\zeta_1)\d\zeta_1
\]
and using the fact that, for all $J<\tilde J_\infty$, $J=\tilde J_M$ for some $M\geq0$, we come to the functional

\begin{equation}
\tilde{ \Fcal}_{\tilde J_M,L}(E)=\frac{1}{L^d}\Big[\per_1(E,[0,L)^d)\int_{\R^{d-1}}\int_{-M}^M |\zeta_1|{K}_1(\zeta) \d\zeta-\int_{[0,L)^d}\int_{\R^d}|\chi_E(x+\zeta)-\chi_E(x)|K_1(\zeta)\d\zeta\dx\Big].
\end{equation}

\subsection{Energy and width of optimal stripes}
\label{ss:ew}
We are interested in showing pattern formation for $J=\tilde J_M$, with $M$ large but finite.  Since we expect the width of the stripes to become larger and larger and the value of the functional to approach $0$ as $M$ tends to $+\infty$, it is convenient to find the optimal energy and width among all the $[0,L)^d$-periodic stripes so as to find suitable  rescaling parameters.

Let 
\[
E_h=\bigcup_{j=1}^N((2j-1)h, 2jh)\times\R^{d-1}
\]
be a periodic union of stripes of width and distance $h$ (in particular, $L=Nh$) and assume that 
\[
1\ll 2M\leq L, 
\] 

The energy of $E_h$, that we denote by $e_{M}(h)$, is given by the following formula:
\begin{align}
e_{M}(h)&=\tilde{ \Fcal}_{\tilde J_M,L}(E_h)={2}\Big[\frac{1}{h}\int_{0}^M\zeta_{1}\widehat K_1(\zeta_1)\d\zeta_1-\frac{1}{h}\int_{0}^{+\infty}\min\{h,\zeta_1\}\widehat K(\zeta_1)\d\zeta_1\notag\\
&-\frac{1}{h}\sum_{k\in\N}\int_0^h\int_{2kh}^{(2k+1)h}\widehat K_1(u-v)\dv\du\Big].\label{eq:ehm}
\end{align}
In particular, if $h<M$ then
\begin{equation}
\label{eq:e<m}
e_{M}(h)=\frac1h\int_h^M\zeta_1\widehat K_1(\zeta_1)\d\zeta_1-\int_h^{+\infty}\widehat K_1(\zeta_1)\d\zeta_1-\frac{1}{h}\sum_{k\in\N}\int_0^h\int_{2kh}^{(2k+1)h}\widehat K_1(u-v)\dv\du,
\end{equation}

while if $h>M$ 

\begin{equation}
\label{eq:e>m}
e_{M}(h)=-\frac{1}{h}\int_M^{+\infty}\min\{h,\zeta_1\}\widehat K_1(\zeta_1)\d\zeta_1 -\frac{1}{h}\sum_{k\in\N}\int_0^h\int_{2kh}^{(2k+1)h}\widehat K_1(u-v)\dv\du. 
\end{equation}

Let us check that 
\[
e_{M}(h)>e_{M}(M),\quad\forall\,h>M.
\]

First of all notice that, since 
\[
\int_M^h\frac{\rho}{h}\widehat K_1(\rho)\d\rho<\int_M^h\widehat K_1(\rho)\d\rho \quad\forall\,h>M,
\]
the first term in \eqref{eq:e>m} is smaller if $h=M$. 

Then we show that, setting
\[
g(h):=\frac1h\sum_{k\in\N}\int_0^h\int_{2kh}^{(2k+1)h}\widehat K_1(u-v)\du\dv,
\]
$g'(h)<0$ if $h>M$. From these two facts one deduces immediately that $M$ is minimal among all $h\geq M$. In the estimates below, one can assume w.l.o.g. that
\begin{equation}
\label{eq:Kest}
\widehat K_1(z)\sim e^{-|z|}.
\end{equation}
One has that
\begin{align}
g'(h)&=-\frac{1}{h^2}\sum_{k\in\N}\int_0^h\int_{2kh}^{(2k+1)h}\widehat K_1(u-v)\du\dv+\frac1h\sum_{k\in\N}\int_{2kh}^{(2k+1)h}\widehat K_1(h-v)\dv\notag\\
&+\frac1h\sum_{k\in\N}(2k+1)\int_0^h\widehat K_1(u-(2k+1)h)\du-\frac1h\sum_{k\in\N}(2k)\int_0^h\widehat K_1(u-(2k)h)\du\notag\\
&\sim-\frac{1}{h^2}\sum_{k\in\N}\int_0^h\int_{2kh}^{(2k+1)h}\widehat K_1(u-v)\du\dv+\sum_{k\in\N}\frac{-(2k-1)e^{-(2k-1)h}+2(2k)e^{-2kh}-(2k+1)e^{-(2k+1)h}}{h}\notag\\
&<0.\label{eq:g'h}
\end{align}

Let us deal now with the case $h<M$. As for the first two terms in \eqref{eq:e<m}, making the same approximation as in \eqref{eq:Kest}, one obtains
\begin{align}
\frac1h\int_h^M\rho\widehat K_1(\rho)\d\rho-\int_h^{+\infty}\widehat K_1(\rho)\d\rho\sim\frac{e^{-h}(h+1)-e^{-M}(M+1)}{h}-e^{-h}\notag\\
\sim \frac{e^{-h}-e^{-M}(M+1)}{h}.
\end{align}
The last term of \eqref{eq:e<m} can be estimated as follows:
\begin{align}
-\frac{1}{h}\sum_{k\in\N}\int_0^h\int_{2kh}^{(2k+1)h}\widehat K_1(u-v)\dv\du&\sim-\sum_{k\in\N}\frac{(e^{-2kh}-e^{-(2k+1)h})(e^h-1)}{h}\notag\\
&\sim-\frac{e^{-h}(1-e^{-h})}{(1+e^{-h})h}.
\end{align}
Therefore, for $h<M$
\begin{align}
e_M(h)\sim\frac{2e^{-2h}-e^{-M}(M+1)-e^{-h}e^{-M}(M+1)}{h(1+e^{-h})}.
\end{align}
Hence  we have the following: for any $C<1/2$, if $h\leq CM$ and $M$ is sufficiently large depending on $C$, then $e_M(h)>0$, thus $h$ cannot be optimal. 
Finally, one has that
\begin{align}
e'_M(h)\sim\frac{-4e^{-2h}h(1+e^{-h})-2e^{-2h}+e^{-M}(M+1)[-(1+e^{-h})^2+3h(1+e^{-h})e^{-h}]}{h^2(1+e^{-h})^2},
\end{align}
which is negative if $h\in[M/2,M)$ and $M$ is large enough.

Therefore, if $h_M$ is the optimal width, then 

\[
 h_M\leq M\quad\text{ and }\quad h_M/M\to1 \text{ as }M\to+\infty.
\]

As a consequence, the following holds.

	\begin{equation}
		\label{eq:lemma_stima_energia}
			e_{M}(h) \sim - e^{-\alpha_M M}
	\end{equation}
	for some $\alpha_M\leq1$, $\alpha_M\to1$ as $M\to+\infty$.

\subsection{Rescaling}
\label{ss:rescaling}

Let us denote by $h_{M,L}$ an optimal width and distance and by $e^*_M$ the minimal energy for the functional $\Fcal_{\tilde J_M,L}$.

As we already saw

\begin{equation*}
\begin{split}
M/2 \leq h_{M,L}  \leq M  \qquad\text{and} \qquad e^*_{M} \geq -e^{-\alpha_M M}
\end{split}
\end{equation*}
where $\alpha_M \to 1$ and $h_{M,L}/M \to 1$ as $M\to+\infty$ for $1<< M << L$.

In this section we will rescale the spacial variables and the functional  so that the optimal width and distance for unions of stripes is $O(1)$ and  the energy is $O(1)$.

By the change of variables $M \zeta' = \zeta $, $Mx'=x$ we have that
\begin{equation*}
\begin{split}
&\per_1(E/M, [0,L/M)^d) =  M^{1-d} \per_1(E, [0,L)^d), \\ 
&J_M := -\int_{\R^{d-1}} \int_{-1}^1  \frac{|\zeta'_1|}{e^*_M|\zeta' |^{d-2}} \exp(-M|\zeta'|) \d\zeta' 
= {-e^*_M}^{-1}M^{-3} \int_{\R^{d-1}} \int_{-M}^M |\zeta_1 | K_1(\zeta) \d\zeta   = {-e^*_M}^{-1}M^{-3}\tilde{J}_M\\
&\text{and}
\\
&\int_{[0,L)^d} \int_{\R^d} |\chi_{E}(x+\zeta) - \chi_E(x) | K_1(\zeta)\dx\d\zeta = M ^{d+2} \int_{[0,L/M)^d} \int_{\R^d} \frac{|\chi_{E/M}(x' + \zeta') -\chi_{E/M}(x')|}{ |\zeta' |^{d-2} } e^{-M|\zeta' |} \dx'\d\zeta'.
\end{split} 
\end{equation*}

Thus we have that
\begin{equation*}
\begin{split}
\tilde{J}_M  \per_1(E,[0,L)^d) = -M^{d+2} J_M e^*_M\per_1(E/M, [0,L/M)^d).
\end{split}
\end{equation*}

Finally defining
\begin{equation*}
\begin{split}
\bar K_M(\zeta)  = \frac{-1}{{e^*_M} |\zeta|_1^{d-2}} e^{-M|\zeta |_1}
\end{split}
\end{equation*}
and putting everything together we have that
\begin{equation*}
\begin{split}
\tilde \Fcal_{\tilde{J}_M,L}(E)  = \frac{-M^{d+2}e^*_M}{L^d}\Big( J_M \per_1(E/M, [0,L/M)^d) - \int_{[0,L/M)^d} \int_{\R^d} {|\chi_{E/M}(x' + \zeta') -\chi_{E/M}(x')|} \bar K_M(\zeta') \dx'\d\zeta'   \Big ).
\end{split}
\end{equation*}

Then let us set
\begin{equation}
\tilde{ \Fcal}_{\tilde J_M,L}(E)=-e^*_M\Fcal_{M, \tilde L}(\tilde E)
\end{equation}
where $\tilde L=L/M$ and $\tilde E=E/M$, and let us drop the tildes in the r.h.s.. Hence
\begin{equation}
\Fcal_{M,L}(E)=\frac{M^2}{L^d}\Big(J_M\per_1(E,[0,L)^d)-\int_{[0,L)^d}\int_{\R^d}|\chi_E(x+\zeta)-\chi_E(x)|\bar K_M(\zeta)\dx\d\zeta\Big).
\end{equation}

 Notice also that, due to \eqref{eq:lemma_stima_energia},  there exists a constant $1>\gamma_M > 0$ such that 
\begin{equation}\label{eq:kmest}
\begin{split}
\bar K_{M}(\zeta) \geq \frac{1}{|\zeta|_1^{d-2}} e^{ -M (|\zeta|_1  - \gamma_M) }
\end{split}
\end{equation}
with $\gamma_M\to 1$ as $M\to+\infty$.

For simplicity of notation we define 
\begin{equation}\label{eq:kmh}
\begin{split}
\widehat {\bar K}_M(t) := \int_{\R^{d-1}} \bar K_M(t,\zeta_2,\ldots,\zeta_d) \d\zeta_2\cdots\d\zeta_d. 
\end{split}
\end{equation}

\subsection{Splitting and lower bounds}

Using the equality
\begin{equation*}
   \begin{split}
|\chi_E(x)-\chi_E(x+\zeta)| = &|\chi_E(x)-\chi_E(x+\zeta_i)|+|\chi_E(x+\zeta)-\chi_E(x+\zeta_i)| \\ &-2|\chi_E(x)-\chi_E(x+\zeta_i)||\chi_E(x+\zeta)-\chi_E(x+\zeta_i)|
   \end{split}
\end{equation*}
one splits the nonlocal term  getting the following lower bound:

\begin{align}
\Fcal_{M,L}(E)&\geq\frac{M^2}{L^d}\sum_{i =1}^d\Big[\int_{[0,L)^d \cap\partial E}\int_{ \R^{d-1}}\int_{-1}^1|\nu^E_i(x)||\zeta_i|\bar K_M(\zeta)\d\zeta\d\mathcal H^{d-1}(x)\notag\\
&-\int_{[0,L)^d}\int_{ \R^d}|\chi_E(x+\zeta_i)-\chi_E(x)|\bar K_M(\zeta)\d\zeta\dx\Big]\notag\\
&+\frac{2}{d}\frac{M^2}{L^d}\sum_{i =1}^d\int_{[0,L)^d}\int_{ \R^d}|\chi_E(x+\zeta_i)-\chi_E(x)||\chi_E(x+\zeta_i^\perp)-\chi_E(x)|\bar K_M(\zeta)\d\zeta\dx\notag\\
&=\frac{M^2}{L^d}\Big(\sum_{i=1}^d\mathcal G^i_{M,L}(E)+\sum_{i=1}^d I^i_{M,L}(E)\Big),\label{eq:yukawa_gstr1}
\end{align}

where 
\[
\mathcal G^i_{M,L}(E):=\int_{[0,L)^d \cap\partial E}\int_{ \R^{d-1}}\int_{-1}^1|\nu^E_i(x)||\zeta_i|\bar K_M(\zeta)\d\zeta\d\mathcal H^{d-1}(x)-\int_{[0,L)^d}\int_{ \R^d}|\chi_E(x+\zeta_i)-\chi_E(x)|\bar K_M(\zeta)\d\zeta\dx
\]
and 
\[
I^i_{M,L}(E):=\frac{2}{d}\int_{[0,L)^d}\int_{ \R^d}|\chi_E(x+\zeta_i)-\chi_E(x)||\chi_E(x+\zeta_i^\perp)-\chi_E(x)|\bar K_M(\zeta)\d\zeta\dx.
\]
Moreover, let 
\[
I_{M,L}(E):=\sum_{i=1}^dI^i_{M,L}(E).
\]

One can further express $\mathcal G^i_{M,L}(E)$ as a sum of contributions obtained by first slicing and then considering interactions with neighbouring points on the slice lying on $\partial E$, namely

\begin{equation}
\mathcal G^i_{M,L}(E)=\int_{[0,L)^{d-1}}\sum_{s \in \partial E_{t^\perp_i}\cap[0,L)}r_{i,M}(E,t_i^\perp,s)\dt_i^\perp
\end{equation}
where for $s\in\partial E_{t_i^\perp}$
\begin{align}
r_{i,M}(E,t_i^\perp,s):=\int_{-1}^1 |\zeta_{i}| \widehat{\bar K}_M(\zeta_{i})\d\zeta_i &- \int_{s^-}^{s}\int_{0}^{+\infty} |\chi_{E_{t_{i}^{\perp}}}(u + \rho) - \chi_{E_{t_{i}^{\perp}}}(u) | \widehat{\bar K}_M(\rho)\d\rho \du \notag \\ & - \int_{s}^{s^+}\int_{-\infty}^{0} |\chi_{E_{t_{i}^{\perp}}}(u + \rho) - \chi_{E_{t_{i}^{\perp}}}(u) |\widehat{\bar K}_M(\rho) \d\rho \du\label{eq:rimdef}
\end{align}
and
\begin{equation}
\label{eq:yukawa_s+s-}
\begin{split}
s^+ &= \inf\{ t' \in \partial E_{t_{i}^{\perp}}, \text{with } t' > s  \} \\ s^- &= \sup\{ t' \in \partial E_{t_{i}^{\perp}}, \text{with } t' < s  \}. 
\end{split}
\end{equation}

Setting
\begin{equation}\label{eq:yukawa_fE}
\begin{split}
f_E(t^\perp_i,t_i,t'^\perp_i,t'_i):=|\chi_E(t_i^\perp +t_i+ t'_i) - \chi_E(t_i + t^\perp_i ) | | \chi_{E}(t_i^\perp +t_i+ t'^\perp_i) - \chi_E(t_i + t^\perp_i)  |,
\end{split}
\end{equation}
we can rewrite the last term in the r.h.s. of~\eqref{eq:yukawa_gstr1} as
 
\begin{align}
\label{eq:yukawa_decomp_double_prod}
I^i_{M,L}(E)=\frac{2}{d} \int_{[0,L)^d } \int_{\R^d} f_{E}(t_{i}^{\perp},t_{i},\zeta_{i}^{\perp},\zeta_{i}) \bar K_{M}(\zeta)  \d\zeta \dt &= \int_{[0,L)^{d-1}} \sum_{s\in \partial E_{t_{i}^{\perp}}\cap [0,L)} v_{i,M}(E,t_{i}^{\perp},s)\dt_{i}^\perp \notag\\
&+ \int_{[0,L)^d} w_{i,M}(E,t_i^\perp,t_i) \dt
\end{align}

where
\begin{equation}\label{eq:witau}
{w}_{i,M}(E,t_{i}^{\perp},t_{i}) = \frac{1}{d}\int_{\R^d}  
f_{E}(t_{i}^{\perp},t_{i},\zeta_{i}^{\perp},\zeta_{i}) \bar K_{M}(\zeta)  \d\zeta. 
\end{equation}
and
\begin{equation}\label{eq:vitau}
v_{i,M}(E,t_{i}^{\perp},s) =  \frac{1}{2d}\int_{s^{-}}^{s^{+}} \int_{\R^{d}} f_{E}(t^{\perp}_{i},u,\zeta^{\perp}_{i},\zeta_{i}) \bar K_{M}(\zeta) \d\zeta\du
\end{equation}
and $s^+,s^-$ as in \eqref{eq:yukawa_s+s-}. 

Notice that the term $r_{i,M}$ is  penalizing sets whose slices in direction $i$ have boundary points at distance smaller than some given constant (see Lemma \ref{lemma:rim}).

The term $v_{i,M}$  penalizes oscillations in direction $e_i$ whenever the neighbourhood of the point $(t_{i}^{\perp}+se_i)$ is close in $L^1$ to a stripe oriented along $e_j$. 
This  statement is made precise in Lemma~\ref{lemma:yukawa_stimaContributoVariazionePiccola}.

\subsection{Averaging}
\label{ss:average}

We now define the ``local contribution'' to the energy in a cube $Q_l(z)$, where $z\in[0,L)^d$ and $0<l<L$ is a length which in the Section \ref{sec:yukawa_taul} will be fixed independently of $L$ ($l$ will depend only on the dimension).

We let 

 \begin{equation} 
\label{eq:yukawa_fbartau}
\begin{split}
\bar{F}_{i,M}(E,Q_{l}(z)) &:= \frac{1}{l^d  }\Big[\int_{Q^{\perp}_{l}(z_{i}^{\perp})} \sum_{\substack{s \in \partial E_{t_{i}^{\perp}}\\ t_{i}^{\perp}+se_i\in Q_{l}(z)}} (v_{i,M}(E,t_{i}^{\perp},s)+ r_{i,M}(E,t_{i}^{\perp},s)) \dt_{i}^{\perp} + \int_{Q_{l}(z)} {w_{i,M}(E,t_{i}^{\perp}, t_i) \dt}\Big],\\
\bar{F}_{M}(E,Q_{l}(z)) &:= \sum_{i=1}^d\bar F_{i,M}(E,Q_{l}(z)).
\end{split}
\end{equation}

Thanks to Lemma 7.2 in \cite{DR}, one has that the r.h.s. of~\eqref{eq:yukawa_gstr1} is equal to

 \begin{equation}
\label{eq:yukawa_gstr15}
\frac{M^2}{L^d}\int_{[0,L)^d}\bar{F}_{M}(E,Q_{l}(z))\dz. 
\end{equation}

 This implies that 
\begin{equation}
\label{eq:yukawa_gstr14}
\begin{split}
\Fcal_{M,L}(E) \geq \frac{M^2}{L^d} \int_{[0,L)^d}  
\bar{F}_{M}(E,Q_{l}(z)) \dz. 
\end{split}
\end{equation}
Given that, in the above inequality, equality holds for stripes, if we show that the minimizers of \eqref{eq:yukawa_gstr15} are periodic optimal stripes, then the same claim holds for $\Fcal_{M,L}$. 

\subsection{A distance from unions of stripes}

The purpose of the next definition is to introduce a quantity which measures the distance of a given set $E$ from being a union of stripes.



\begin{definition}
	\label{def:yukawa_defDEta}
	For every $\eta>0$ we denote by $\Acal^{i}_{\eta}$ the family of all sets $F$ such that  
	\begin{enumerate}[(i)]
		\item they are union of stripes oriented along the direction $e_i$ 
		\item their connected components of the boundary are distant at least $\eta$. 
	\end{enumerate}
	We denote by 
	\begin{equation} 
	\label{eq:yukawa_defDEta}
	\begin{split}
	D^i_\eta(E,Q) := \inf\Big\{ \frac{1}{\vol(Q)} \int_Q|\chi_E -\chi_F|:\ F\in \mathcal{A}^{i}_\eta \Big\} \quad\text{and}\quad D_\eta(E,Q) = \inf_i D^i_\eta(E,Q).
	\end{split}
	\end{equation} 
	Finally, we let $\mathcal A_\eta:=\cup_{i}\mathcal A^i_{\eta}$.
\end{definition}


We recall now some properties of the distance \eqref{eq:yukawa_defDEta} (see Remark 7.4 in \cite{DR}). 


\begin{lemma}
	\label{rmk:yukawa_lip} \ 
	\begin{enumerate}[(i)]
		\item Let $E, F \subset \mathbb{R}^d$.  Then, the map $z\mapsto D_\eta(E,Q_l(z))$ is Lipschitz. The Lipschitz constant can be shown to be $C_d/l$, where $C_d$ is a constant depending only on the dimension $d$.

		
		\item For every $\varepsilon > 0$,  there exists $\delta_0= \delta_0(\varepsilon)$,  such that for every $\delta \leq \delta_0$ and  whenever $D^j_{\eta}(E,Q_{l}(z))\leq \delta$ and $D^i_\eta(E,Q_l(z))\leq \delta$ for  $i\neq j$ and $\eta>0$,  the following hold
		\begin{equation}
		\label{eq:yukawa_gsmstr2}
		\begin{split}
		\min\big(|Q_l(z)\setminus E|, |E \cap Q_l(z)| \big) \leq \varepsilon. 
		\end{split}
		\end{equation}
	\end{enumerate}
\end{lemma}

\section{The one-dimensional problem}

\label{sec:yukawa_1D_problem}
We consider the following one-dimensional functional:  on an $L$-periodic set $E\subset\R$ of locally finite perimeter
\[
\Fcal_{M,L}^1(E)=\frac{M^2}L\Big(\int_{-1}^1\widehat {\bar K}_M(\rho)\Big[\per(E, [0,L))|\rho|-\int_0^L|\chi_E(s)-\chi_E(s+\rho)|\ds\Big]\d\rho\Big),
\]
where $\widehat{\bar K}_M$ has been defined in \eqref{eq:kmh}.

The functional $\Fcal^1_{M,L}$ corresponds to $\Fcal_{M,L}(E)$ when the set $E$ is a union of stripes.  Namely, given $E \subset \R^d$ and such that $E = \hat E \times \R^{d-1}$ where $E$ is $L$-periodic, then 
\begin{equation*}
   \begin{split}
      \Fcal_{M,L}(E) =  \Fcal^1_{M,L}(\hat E).
   \end{split}
\end{equation*}

The purpose of this section is to show that the periodic sets are minimizers among the sets composed of stripes, whenever $M$ is large enough. 
For the above one-dimensional problem there are some standard techniques available in the literature. 
In particular, our proof will rely on the  reflection positivity technique, introduced in the context of quantum field theory by Osterwalder and Schrader and then applied for the first time in statistical mechanics by Fr\"ohlich, Simon and Spencer.
For works where the reflection positivity is used in models with competing interactions, see e.g. \cite{fro}, \cite{glllRP1,glllRP3,glllRP3bis,glllRP2}, 
\cite{GR,DR}.
As the  technique is nowadays standard, we will only outline briefly the steps and show some of the differences with respect to the literature.

Before showing optimality of periodic stripes, we show that there exists a unique optimal period $h^*_M$, provided $M$ is large enough. For $h>0$, let $E_h=\cup_{j\in\Z}[(2j)h,(2j+1)h]$ and define the energy
\[
e_{M}(h)={ \Fcal^1}_{ M,2h}(E_h)=\frac{M^2}{h}\Big[\int_{h}^1(\rho-h)\widehat {\bar K}_M(\rho)\d\rho-\int_{1}^{+\infty}h\widehat {\bar K}_M(\rho)\d\rho-\sum_{k\in\N}\int_0^h\int_{2kh}^{(2k+1)h}\widehat {\bar K}_M(u-v)\dv\du\Big].
\]

We prove the following

   \begin{theorem}\label{thm:unique}
      There exists $\tilde M>0$ such that $\forall\,M>\tilde M$, there exists a unique minimizer of $e_M(\cdot)$, $h^*_M$. 
   \end{theorem}

   \begin{proof}
      In Section \ref{ss:ew} we showed that, as $M\to+\infty$, minimizers to $e_M$ tend to $1$. So, in order to prove the theorem it is sufficient to show that there exist $\delta>0$, $\tilde M$ such that $e''_M(1)\geq\delta$ for all $M>\tilde M$.   Computing $e''_M(1)$, one finds
      \begin{align}
         \frac{e''_M(1)}{M^2}&=\widehat{\bar K}_M(1)\notag\\
         &-2\sum_{k\in\N}\int_0^1\int_{2k}^{2k+1}\widehat{\bar K}_M(u-v)\dv\du+2\sum_{k\in\N}\int_{2k}^{2k+1}\widehat{\bar K}_M(h-v)\dv\notag\\
         &+2\sum_{k\in\N}\int_0^1[(2k+1)\widehat{\bar K}_M(u-(2k+1))-2k\widehat{\bar K}_M(u-2k)]\du\notag\\
         &-2\sum_{k\in\N}[(2k+1)\widehat{\bar K}_M(2k)-2k\widehat{\bar K}_M(2k-1)]\notag\\
         &-\sum_{k\in\N}\int_{2k}^{2k+1}\widehat{\bar K}'_M(h-v)\dv\notag\\
         &+\sum_{k\in\N}\int_0^1[(2k+1)^2\widehat{\bar K}'_M(u-(2k+1))-(2k)^2\widehat{\bar K}'_M(u-2k)]\du.         
      \end{align}
      One has the following lower estimate:
      \begin{align}
       \frac{e''_M(1)}{M^2}&\geq2\sum_{k\in\N}[(2k+1)\widehat{\bar K}_M(2k+1)-(2k)\widehat{\bar K}_M(2k)-(2k+1)\widehat{\bar K}_M(2k)-(2k)\widehat{\bar K}_M(2k-1)\notag\\
       &+M\sum_{k\in\N}\widehat{\bar K}_M(2k-1)\notag\\
       &-M\sum_{k\in\N}[(2k+1)^2\widehat{\bar K}_M(2k+1)-(2k)^2\widehat{\bar K}_M(2k)].
      \end{align}
      
      To conclude, observe now that the lower bound can be rewritten as 
      \begin{align}
      \sum_{k\in2\N}e^{-Mk}\{M[-(k+1)^2e^{-M}+k^2+e^M]+2[(k+1)e^{-M}-(2k+1)+ke^M]\},
      \end{align} 
      which is positive for $M$ sufficiently large.

   \end{proof}

Let us now return to the issue of showing that optimal periodic stripes are optimal among all stripes. This fact follows from the reflection positivity technique and is well-known in the literature. 
We refer the reader to the references given at the beginning of this section. 
Only for notational reasons we also refer to \cite{DR,GR}.
The only part needed is the reflection positivity of the Yukawa kernel. 


\begin{lemma}
 \label{lemma:yukawa_laplace}
 The Yukawa kernel is reflection positive, namely there exists a positive Borel measure  $\mu$ such that
 \begin{equation*}
    \begin{split}
       \widehat{K}_1(t) = \int_0^{+\infty} e^{-\alpha t} \d\mu(\alpha).
    \end{split}
 \end{equation*}
\end{lemma}

\begin{proof}
   Due to the Hausdorff-Bernstein-Widder theorem, one has that a function $\varphi$ is reflection positive if and only if it is completely monotone, namely $(-1)^n \frac {\d^n}{dt^n} \varphi \geq 0$. 

   By using the complete monotone property one has that the set of functions which are reflection positive is an algebra. Indeed, given two functions $\varphi$, $\psi$ which are completely monotone by the Leibniz rule one has that
   \begin{equation*}
      \begin{split}
         (-1)^n \frac {\d^n}{\dt^n} (\varphi \psi) =  \sum_{k=0}^n (-1)^k \frac {\d^k\varphi}{\dt^k} \cdot (-1)^{n-k}\frac {\d^{n-k}\psi}{\dt^{n-k}}  \geq 0
      \end{split}
   \end{equation*}

   In order to conclude the proof, we need to show that the map 
   \begin{equation*}
      \begin{split}
         t\mapsto e^{-t} \int_{\R^{d-1}} \frac1{ (t + |\zeta_2| + \cdots + |\zeta_d|)^{d-2} } e^{ -(|\zeta_2| + \cdots + |\zeta_d|) }
      \end{split}
   \end{equation*}
   is completely monotone.  Given the complete monotone functions are an algebra, we need to check that the single terms in the product above are completely monotone. This can be done easily with explicit calculations.


\end{proof}

%

Once reflection positivity of the kernel is shown the proof follows by standard means in the literature. We refer the reader to \cite{GR,DR} where further details are given and a similar notation is used.

\section{A local rigidity estimate}
\label{sec:yukawa_taul}

The core of the proof of Theorem \ref{thm:main_yukawa} is the following proposition (see alse \cite{DR} for a similar proposition in a different setting):

\begin{proposition}[Local Rigidity] 
	\label{lemma:yukawa_local_rigidity}
	For every $N > 1,l,\delta > 0$, there exist $\bar{M},\bar{\eta} >0$ 
	such that whenever $M> \bar{M}$  and $\bar F_{M}(E,Q_{l}(z)) < N$ for some $z\in [0,L)^d$ and $E\subset\R^d$ $[0,L)^d$-periodic, with $L>l$, then it holds $D_{\eta}(E,Q_{l}(z))\leq\delta$ for every $\eta < \bar{\eta}$. Moreover $\bar{\eta}$ can be chosen independent on $\delta$.  Notice that $\bar{M}$ and $\bar{\eta}$ are independent of $L$.
\end{proposition}

In particular, such a rigidity estimate tells us that on small cubes  minimizers of the functional are, for $M$ large enough, close to stripes of a given minimal width.

In order to prove Proposition \ref{lemma:yukawa_local_rigidity}, we will need to analize the behaviour of $\bar F_M$ for large $M$. 
First of all, we start with the following lemma, about the term $r_{i,M}$. In particular this tells us that given a sequence of sets $\{E_M\}_{M>0}\subset\R^d$ of bounded local energy $\bar F_{M}(E, Q_l(z))$, if $M$ is large enough their boundary points on the slices have distance at least $\eta_0$ where $\eta_0$ is fixed arbitrarily close to $1$ and then they converge to a set of locally finite perimeter $E_0$.

 \begin{lemma}\label{lemma:rim} 
 		 There exists a function $g:\R\times\R\to\R$ such that, for all $E\subset\R^d$ of locally finite perimeter, $t_i^\perp\in[0,L)^{d-1}$, $s\in\partial E_{t_i^\perp}$ 
 	\begin{equation}\label{eq:rimest}
 	r_{i,M}(E,t_i^\perp,s)\geq g\bigl( (\gamma_M-\min(|s-s^-|,|s-s^+|),M\bigr)
 	\end{equation}
 	with $\gamma_M$ defined in \eqref{eq:kmest}. The function $g$ satisfies the following: $g(v,M)\geq g(v',M)$ whenever $v>v'$, $g(v,M)\geq -e^{-cM}$ for some $c>0$ and $g(v,M)\to +\infty$ as $M\to+\infty$ provided $v>0$.
 	
 	In particular, for every $0<\eta_0<1$ there exists $M_0$ such that, for all $M>M_0$ if $\min\{|s-s^-|,|s-s^+|\}<\eta_0$ then $r_{i,M}(E,t_i^\perp,s)>0$.
 	\end{lemma}

 \begin{proof}
 	The proof of this lemma uses the following inequality: for every $E\subset\R^d$ of locally finite perimeter, $\forall\,t_i^\perp\subset[0,L)^{d-1}$,  
 	
 	\begin{equation}\label{toprovebandesim}
 	\begin{split}
 	&\forall\,\rho>0,\quad\int_{s^-}^s |\chi_{E_{t_i^\perp}}(u)-\chi_{E_{t_i^\perp}}(u+\rho)| \du\leq  \min(\rho,s-s^-)\\
 	&\forall\,\rho<0,\quad\int_{s}^{s^+} |\chi_{E_{t_i^\perp}}(u)-\chi_{E_{t_i^\perp}}(u+\rho)| \du\leq  \min(-\rho,s^+-s).
 	\end{split}
 	\end{equation}
 	
 	Indeed,
 	
 	\begin{align}
 	\int_0^1\zeta_i\widehat {\bar K}_M(\zeta_i)\d\zeta_i&-\int_{s^-}^s\int_0^{+\infty}|\chi_{E_{t_i^\perp}}(u+\rho)-\chi_{E_{t_i^\perp}}(u)|\widehat {\bar K}_M(\rho)\d\rho\du\notag\\
 	&\overset{\eqref{toprovebandesim}}{\geq}\int_0^1\zeta_i\widehat {\bar K}_M(\zeta_i)\d\zeta_i-\int_0^{+\infty}\min(|s-s^-|,\zeta_i)\widehat {\bar K}_M(\zeta_i)\d\zeta_i\notag\\
 	&\geq\int_{|s-s^-|}^1(\zeta_i-|s-s^-|)\widehat {\bar K}_M(\zeta_i)\d\zeta_i-\int_1^{+\infty}|s-s^-|\widehat {\bar K}_M(\zeta_i)\d\zeta_i
 	\end{align}
 	
 	and analogously
 	\begin{align}
 	\int_{-1}^0|\zeta_i|\widehat {\bar K}_M(\zeta_i)\d\zeta_i&-\int_{s}^{s^+}\int_{-\infty}^0|\chi_{E_{t_i^\perp}}(u+\rho)-\chi_{E_{t_i^\perp}}(u)|\widehat {\bar K}_M(\rho)\d\rho\du\notag\\
 	&\geq\int_{-1}^{-|s-s^+|}(|\zeta_i|-|s-s^+|)\widehat {\bar K}_M(\zeta_i)\d\zeta_i-\int_{-\infty}^{-1}|s-s^+|\widehat {\bar K}_M(\zeta_i)\d\zeta_i.
 	\end{align}
 	
 	Then, by \eqref{eq:rimdef} and  using \eqref{eq:kmest}, since $\gamma_M<1$ one gets \eqref{eq:rimest} and the statement of the lemma. 
 \end{proof}
 
 \begin{remark}\label{rmk:rim}
 From Lemma \ref{lemma:rim}, since $\gamma_M\to1$ as $M\to+\infty$, it follows as well that the function
 \[
 r_{i,\infty}(E,t_i^\perp,s):=\liminf_{M\to+\infty}r_{i,M}(E, t_i^\perp,s)
 \]
 satisfies 
 \begin{equation}\label{eq:riest}
 r_{i,\infty}(E, t_i^\perp,s)=+\infty\quad\text{whenever}\quad \min(|s-s^-|,|s-s^+|)<1.
 \end{equation}
 In particular, if $\{E_{M}\}_{M>0}\subset \R^d$  is a family of sets of locally finite perimeter with $\sup_M\bar F_M(E_M, Q_l(z))\leq N$, then for a.e. $t_i^\perp\in Q_l^\perp(z_i^\perp)$ and for every $I\subset \R$ open interval,
 \begin{equation}\label{eq:rm1}
 \liminf_{M\to+\infty}\min \bigl\{|s^M_i-s^M_{i+1}|:\, \partial E_{M,t_i^\perp}\cap I=\{s^M_i\}_{i=1}^{m(M)}\bigr\}\geq1
 \end{equation}
 In particular, $E_M$ converges in $L^1_{\mathrm{loc}}$ to a set $E_\infty$ of locally finite perimeter such that
 \begin{equation}\label{eq:rm2}
 \min\bigl\{|s^\infty_i-s^\infty_j|:\,\partial E_{\infty,t_i^\perp}\cap I=\{s^\infty_k\}_{k=1}^{m(\infty)}\bigr\}\geq 1.
 \end{equation}
 For details of how to deduce \eqref{eq:rm2} form \eqref{eq:rm1} see the proof of Lemma 7.5 in \cite{DR}.
 
 \end{remark}

 \begin{remark}
 	Let us notice the following: the family of kernels ${\bar K}_M$ is monotone increasing in $M$ as $M\to+\infty$. 
 		Let ${\bar K}_\infty$ be defined by
 	\begin{equation}\label{eq:kidef}
 	{\bar K}_{\infty}(\zeta):=\liminf_{M\to +\infty}{\bar K}_M(\zeta).
 	\end{equation}
 	From \eqref{eq:kmest} we get that 
 	\begin{equation}\label{eq:kiest}
 	{\bar K}_\infty\geq0\quad\text{and} \quad {\bar K}_\infty(\zeta)=+\infty\text{ whenever }|\zeta|<1,
 	\end{equation}
 	where the last statement comes from the fact that $\gamma_M\to1$ as $M\to+\infty$.
 \end{remark}

    Let us now proceed to the proof of Proposition \ref{lemma:yukawa_local_rigidity}. 
    The main steps can be summarized as follows.  
    Given a sequence of sets $E_M\subset\R^d$ of bounded local energy, by Remark \ref{rmk:rim} their boundary points on the slices are not too close (they have mutual distance at least $1$) and then they converge to a set of locally finite perimeter $E_\infty$. 
    Then, using the monotonicity in $M$ of the kernel one gets as $M\to+\infty$ a bound on the liminf  of the cross interaction terms $I_{M,L}$ on $E_\infty$ (see \eqref{eq:yukawa_gstr11}). Thanks to the fact that boundary points on the slices of $E_\infty$ have mutual distance at least $\eta_0$ with $\eta_0$ close to $1$ and that \eqref{eq:kiest} holds, one gets that boundary points in $\partial E_{\infty,t_i^\perp}$  are a constant function of $t_i^\perp$. Therefore the only shapes admissible for $E_{\infty}$ are checkerboards or stripes, and finally by an analogous energetic argument we rule out checkerboards.

    As a consequence, for $M$ sufficiently large but depending only on $l$, the sets $E_M$ will be close to $E_\infty$ in the sense of Definition \ref{def:yukawa_defDEta} and therefore to stripes of a minimal given width.


\begin{proof} 
	
	The proof will follow by contradiction. 
	Indeed, assume that the claim is false. 
   This implies that there exists $N > 1, l$, $\delta>0$ and sequences $\{ M_k\}$, $\{ \eta_k\}$, $\{L_k\}$, $\{z_k\}$, $\{E_{M_k}\}$ such that:
	\begin{enumerate}[(i)]
		\item one has that   $M_k\to+\infty $, $L_k > l$, $\eta_k \downarrow 0$, $z_k\in[0,L_k)^d$; 
		\item the family of sets $E_{M_k}$ is $[0,L_k)^d$-periodic
		\item one has that $D_{\eta_k}(E_{M_k},Q_{l}(z_k)) > \delta$ and $\bar{F}_{M_k}(E_{M_k},Q_{l}(z_k)) < N$. 
	\end{enumerate}

	Given that $\eta \mapsto D_{\eta}(E,Q_l(z)) $ is monotone increasing, we can fix $\bar{\eta}$  sufficiently small instead of $\eta_{k}$ with $D_{\bar\eta}(E_{M_k},Q_l(z_k)) >  \delta$. 
	In particular, $\bar{\eta}$ will be chosen at the end of the proof depending only on $N,l$. 
	
	W.l.o.g. (taking e.g. $E_{M_k}-z_k$ instead of $E_{M_k}$) we can assume there exists $z\in\R^d$ such that $z_k=z$ for all $k\in\N$. 
	
	Because of Remark~\ref{rmk:rim}, one has that $\sup_{k}\per_1(E_{M_k},Q_{l}(z)) < +\infty$. 
	Thus up to subsequences there exists $E_{\infty}$ such that $E_{M_{k}} \to E_{\infty}$ in $L^1(Q_l(z))$ with 
	\begin{equation}\label{eq:dee}
	D_{\bar{\eta}}(E_\infty,Q_{l}(z))> \delta
	\end{equation}

	In order to keep the notation simpler, we will write $M\to+\infty$ instead of $M_k \to +\infty$ and  $E_{M}\to E_\infty$ instead of $E_{M_{k}}\to E_\infty $.

	Define $J_i:=(z_{i}-l/2,z_{i}+l/2)$. 
	
	By Lebesgue's theorem, there exists a subsequence of $M$ such that for almost every $t_{i}^{\perp}\in Q_{l}^{\perp}(z_{i}^{\perp})$ one has that $E_{M,t_{i}^{\perp}}\cap J_i$ converges to $E_{\infty,t_{i}^{\perp}}\cap J_i $ in $L^1(Q_l(z))$.

	By using \eqref{eq:yukawa_fbartau} and the fact that $v_{i,M}\geq0$, we have that
	\begin{equation}
	\label{eq:yukawa_gstr7}
	N \geq  \bar{F}_{M}(E_M,Q_l(z)) \geq
	\frac{1}{l^d }\sum_{i =1}^{d}\int_{Q_{l}^{\perp}(z_{i}^{\perp})} \sum_{\substack{s\in \partial E_{M, t^\perp_i}\\ s\in J_i}} 
	r_{i,M}(E_{M},t_{i}^{\perp},s) \dt^\perp_{i} + \int_{Q_{l}(z)} w_{i,M}(E_{M},t_{i}^{\perp},t_{i})\dt_{i}^{\perp}\dt_{i}.
	\end{equation}
	By the Fatou lemma,  we have that 
	\begin{equation*}
	\begin{split}
	l^d M \geq  \liminf_{M\to+\infty}\sum_{i =1}^{d}\int_{Q_{l}^{\perp}(z_{i}^{\perp})} \sum_{\substack{s\in \partial E_{M,t^\perp_i}\\ s\in J_i}}&r_{i,M}(E_{M},t_{i}^\perp,s) \dt_{i}^{\perp}         \geq 
	\sum_{i =1}^{d}\int_{Q_{l}^{\perp}(z_{i}^{\perp})} \liminf_{M\to+\infty}\sum_{\substack{s\in \partial E_{M,t^\perp_i}\\ s\in J_i}} 
	r_{i,M}(E_{M},t_{i}^{\perp},s) \dt^{\perp}_{i}       \\ 
	&\geq 
	\sum_{i =1}^{d}\int_{Q_{l}^{\perp}(z_{i}^{\perp})}\mathbb 1^\infty_{B_i}(t_i^\perp) \dt_{i}^\perp,
	\end{split}
	\end{equation*}
	where
	\[
	B_i=\Big\{t_i^\perp\in Q_l^\perp(z_i^\perp):\,\min\{|s^\infty_i-s^\infty_j|:\,\partial E_{\infty,t_i^\perp}=\{s^\infty_k\}_{k=1}^{m(\infty,t_i^\perp)}\}<1\Big\},
	\]
	and in the last inequality we have used Remark \ref{rmk:rim}. 
	
For the last term  in \eqref{eq:yukawa_gstr7}, namely
	
	\begin{equation}
	\label{eq:yukawa_gstr8}
	\begin{split}
	\liminf_{M\to+\infty}&  \int_{Q_{l}(z)} w_{i,M}(E_{M},t_{i}^{\perp},t_{i})\dt_{i}^{\perp}\dt_{i} 
	\\  
	&\geq \liminf_{M\to +\infty}\frac1d\int_{Q_l(z)}\int_{Q_l(z)}f_{E_M}(t^{\perp}_i,t_i,t'^\perp_{i},t'_i){\bar K}_M(t)\dt\dt'	 \\
	&= \liminf_{M\to+\infty} \frac{1}{d}\int_{Q_{l}(z)} \int_{Q_{l}(z)} f_{E_M}(t^\perp_i, t_{i},t'^{\perp}_{i} - t^\perp_i,t'_{i}- t_i){\bar K}_{M}(t - t')\dt \dt'\\  
	\\   &\geq \frac{1}{d}\int_{Q_{l}(z)} \int_{Q_{l}(z)} f_{E_\infty}(t^\perp_i, t_{i},t'^{\perp}_{i} - t^\perp_i,t'_{i}- t_i) {\bar K}_{\infty}(t - t')\dt \dt',\\  
	\end{split}
	\end{equation}
	where in the third line we have used a change of variables.

	In order to prove \eqref{eq:yukawa_gstr8} we fix $M'>0$ and by using initially $E_{M}\to E_{\infty}$ in $L^1(Q_l(z))$ and afterwards the   monotonicity of $M \mapsto {{\bar K}}_{M}(\zeta)$ we have that
	
	\begin{equation*}
	\begin{split}
	\liminf_{M\to+\infty} \int_{Q_{l}(z)} w_{i,M}(E_{M},t_{i}^{\perp},t_{i}) \dt_{i}^{\perp}\dt_{i} \geq \sup_{M'} 
	\liminf_{M\to+\infty}\int_{Q_{l}(z)} w_{i,M'}(E_{M},t_{i}^{\perp},t_{i}) \dt_{i}^{\perp}\dt_{i}
	\\ \geq  \sup_{M'}\frac{1}{d}\int_{Q_{l}(z)}\int_{Q_{l}(z)} 
	f_{E_{\infty}} (t^\perp_i, t_i, t'^\perp_i - t'^\perp_{i},t_i - t'_i) {\bar K}_{M'}(t-t')\dt \dt'
	\\ \geq \frac1d \int_{Q_l(z)}\int_{Q_l(z)} 
	f_{E_{\infty}} (t^\perp_i, t_i, t'^\perp_i - t'^\perp_{i},t_i - t'_i) {\bar K}_{\infty}(t-t')\dt \dt'.
	\end{split}
	\end{equation*}
	
      Thus, we have shown that 
      \begin{align}
         & \sum_{i=1}^{d}\frac{1}{d}\int_{Q_{l}(z)} \int_{Q_{l}(z)}f_{E_{\infty}} (t^\perp_i, t_i, t'^\perp_i - t'^\perp_{i},t_i - t'_i) {\bar K}_{\infty}(t-t')\dt \dt'  \notag
         \\  &+    \sum_{i =1}^{d}\int_{Q_{l}^{\perp}(z_{i}^{\perp})}\mathbb 1^\infty_{B_i}(t_i^\perp)\dt_{i}^\perp   \lesssim l^d N   \label{eq:gstr381}.
      \end{align}
	

Defining  
\begin{equation}
\label{eq:yukawa_definizioneInteraction}
\begin{split}
\Int(t^\perp_i,t'^\perp_i)   := 
\int_{Q_l^i(z_i)} \int_{Q_l^i(z_i)} f_{E_{\infty}}(t^{\perp}_i,t_i,t'^\perp_{i}-t^{\perp}_i, t'_i-t_i) {\bar K}_{\infty}(t- t')\dt_{i} \dt'_{i},
\end{split}
\end{equation}
one has 

\begin{equation}
\label{eq:yukawa_gstr11}
\begin{split}
 \int_{Q_l^\perp(z_i^\perp)}\int_{Q_l^\perp(z_i^\perp)} \Int(t^{\perp}_{i},t'^{\perp}_{i}) \dt^{\perp}_{i} \dt'^{\perp}_{i}\lesssim l^dN<+\infty
\end{split}
\end{equation}

Given $\lambda\in( 0,\frac l2)$, $u\in (z_i-l+\lambda,z_i+l - \lambda)$ and $t^\perp_{i}\in Q_l^\perp(z_i^\perp)$,  we denote by  
\begin{equation}
\label{eq:yukawa_defri_local}
\begin{split}
r^i_{\lambda}(u, t_{i}^\perp) &:= \min \big\{ \inf \{|u-s |:\ s\in \partial E_{\infty,t^\perp_i} \text{ and } s\in (z_i-l+\lambda,z_i+l - \lambda) \}, |u-z_i+l-\lambda |, |z_i+l-\lambda -u | \big\} \\
r^i_o(t_{i}^\perp) &:= \inf_{s\in \partial E_{\infty,t^\perp_i}\cap Q^i_l(z_i)} \min(s^+-s,s-s^-),
\end{split}
\end{equation}
where $s^+,s^-$ are defined in \eqref{eq:yukawa_s+s-}. 
Notice that, since $\int_{Q_{l}^{\perp}(z_{i}^{\perp})}\mathbb 1^\infty_{B_i}(t_i^\perp)\dt_{i}^\perp<+\infty$, for a.e. $t_i^\perp$ $r^i_o(t_i^\perp)\geq1$.

Notice that the  map $r^i_\lambda(\cdot, t_i^\perp)$ is well-defined for almost every $t^\perp_i$ and measurable. The role of $\lambda>0$ is to deal with the boundary, since $E_{\infty}$ is not $[0,l)^d$-periodic.

Suppose that, for every $u$, one has that $r^{i}_{\lambda}(u,\cdot)$ is constant almost everywhere: if this holds for every $i$, then it is not difficult to see that  $E_\infty$  is (up to null sets) either  a union of stripes or a  checkerboards, where by checkerboards we mean any set whose boundary is the union of affine subspace orthogonal to  coordinate axes, and there are at least two of these directions. 

The checkerboards can be ruled out  via an energetic argument (see the comment at the end of this section).

In order to obtain that $r^i_\lambda(u,\cdot)$ is constant almost everywhere we proceed in the following way.

In the next lemma we give  a lower bound for the interaction term.

\begin{lemma}
	\label{lemma:yukawa_p2d_stimaSlice}
	Let $\lambda \in (0,L/2)$ and let $t'^{\perp}_{i},t^{\perp}_{i}\in Q_l^\perp(z_i^\perp)$, $t_i^\perp \neq t'^\perp_i$ be such that 
	$\min(r^{i}_o(t^{\perp}_i), r^{i}_o(t'^{\perp}_i)) > |t'^\perp_{i} - t^{\perp}_{i}|$ and 
	$|t'^\perp_i - t^\perp_i | \leq \min\{\lambda,1/2\}$.
	Then for every $u\in (z_i-l+\lambda, z_i+l -\lambda)$ it holds
	\begin{equation} 
	\label{eq:yukawa_lemma_AB1}
	\Int(t'^{\perp}_{i},t^{\perp}_{i}) \geq \mathbb 1^\infty_{\{(t'^\perp_i,t_i^\perp):\,r^{i}_\lambda(u,t'^{\perp}_{i})\neq r^{i}_{\lambda}(u,t^{\perp}_{i})\}}(t'^\perp_i,t_i^\perp).
	\end{equation}
\end{lemma}

\begin{proof}

   
   In this lemma we use a slicing argument similar to  Lemma~3.5 in \cite{DR}.
   However the presence of a different kernel gives a different quantitative estimate.

	W.l.o.g. let us assume that $r^{i}_\lambda(u, t^\perp_{i}) < r^i_\lambda(u, t'^{\perp}_i)$.  In particular this implies that $r^i_\lambda(u,t^\perp_i) < \min(|u-z_i+l-\lambda|, | z_i+l - \lambda -u |)$, and hence there exists a point  $s_o \in (z_i-l+\lambda,z_i+l-\lambda)$  
	such that 
	\begin{equation*}
	\begin{split}
	|u - s_o| = \inf \{ |u-s|:\ s\in \partial E_{\infty,t^\perp_i}, s\in (z_i-l+\lambda, z_i+l-\lambda)  \}.
	\end{split}
	\end{equation*}
	
	Let us denote by $\delta = |t'^\perp_i -  t^\perp_i|$ and by $r = |r^i_\lambda(u,t^\perp_i) - r^i_\lambda(u,t'^{\perp}_{i})|$.  
	Given that $r_{o}(t^{\perp}_{i}) > \delta$, the following holds
	\begin{equation*}
	\begin{split}
	(s_{o} -\delta , s_{o} +\delta) \cap E_{\infty,t^{\perp}_i} = (s_{o} , s_{o} +\delta )  \qquad\text{or} \qquad 
	(s_{o} -\delta , s_{o} +\delta) \cap E_{\infty,t^{\perp}_i} = (s_{o} -\delta, s_{o}  )  .
	\end{split}
	\end{equation*}
	Notice that since $\lambda \geq \delta$, we have that $(s_o - \delta, s_o + \delta ) \subset Q_l^i(z_i)$.
   In the following, we will assume that
	\begin{equation}
	\label{eq:yukawa_gstr33}
	\begin{split}
	(s_{o} -\delta , s_{o} +\delta) \cap E_{\infty,t^{\perp}_i} = (s_{o} , s_{o} +\delta ) 
	\end{split}
	\end{equation}
	The other case is analogous.

	We will distinguish two subcases: 
	
	\begin{enumerate}[(i)]
		\item  Suppose $r  > \delta/2$. 
		From the definition of $\delta$ and $r$, for every slice in $t'^\perp_i$ it holds
		\begin{equation*}
		\begin{split}
		(s_{o} - \delta/2 , s_{o} + \delta/2) \cap E_{\infty,t'^{\perp}_i} = (s_{o} - \delta/2 , s_{o} + \delta/2 ) 
		\qquad \text{or}\qquad
		(s_{o} - \delta/2 , s_{o} + \delta/2) \cap E_{\infty,t'^{\perp}_i} = \emptyset. 
		\end{split}
		\end{equation*}
		Indeed on the slice $E_{\infty,t'^\perp_i}$, the closest jump point to $s_o$  is at least $r$ distant and $r>\delta/2$. 
		We will assume the first of the alternatives above. The other case is analogous.  

		For every $a \in (s_o-\delta/2, s_o)$ and $a' \in (s_o, s_o+\delta/2)$, one has that 
		\begin{equation*}
		\begin{split}
		f_{E_\infty}(t^\perp_i, a, t'^\perp_i -  t^\perp_i, a' - a) = 1.
		\end{split}
		\end{equation*} 
		Given that $r^i_{\lambda}(u,t^\perp_i) \leq L $, one hs that
		\begin{equation*}
		\begin{split}
		\Int(t^{\perp}_{i},t'^{\perp}_{i})   & = 
		\int_{Q_l^i(z_i)} \int_{Q_l^i(z_i)} f_{E_\infty}(t^{\perp}_i,t_i,t'^\perp_{i}-t^{\perp}_i, t'_i-t_i) {\bar K}_{\infty}(t'- t)\dt_{i} \dt'_{i} 
		\\ &\geq \int_{s_o-\delta /2}^{s_o} \int_{s_o}^{s_o + \delta /2} f_{E_\infty}(t^{\perp}_i,t_i,t'^\perp_{i}-t^{\perp}_i, t'_i-t_i) {\bar K}_{\infty}(t'- t)\dt'_{i} \dt_{i} 
		\\ &\geq \int_{s_o-\delta /2}^{s_o} \int_{s_o}^{s_o + \delta /2}  {\bar K}_{\infty}(t'- t)\dt'_{i} \dt_{i} 
		=+\infty,
		\end{split}
		\end{equation*}
		since $|t-t'|<1$ if $\max(|t_i^\perp-t_i'^\perp|,|t_i-t_i'|)\leq 1/2$ and therefore $\bar K_\infty(t-t')=+\infty$.
		
		\item  Let us assume now that $r \leq  \delta/2$. 
		Given that $r_{o}(t'^\perp_i),r_o(t^\perp_{i}) > \delta$, one has that
		either
		\begin{equation*}
		\begin{split}
		&(s_{o} - r , s_{o} + \delta/2) \cap E_{\infty,t'^{\perp}_i} = (s_{o} - r , s_{o} + \delta/2 ) \quad\text{or}\quad (s_{o} - r , s_{o} + \delta/2) \cap E_{\infty,t'^{\perp}_i} =\emptyset\\
		\text{or}\quad&(s_{o} - \delta/2 , s_{o} +r) \cap E_{\infty,t'^{\perp}_i} = (s_{o} - \delta/2 , s_{o} + r )\quad\text{or}\quad(s_{o} - \delta/2 , s_{o} +r) \cap E_{\infty,t'^{\perp}_i} =\emptyset.
		\end{split}
		\end{equation*}

		Indeed if none of the above were true we would have that $\#(\partial E_{\infty,t'^{\perp}_{i}} \cap (s_o -\delta/2 , s_o +\delta/2)) \geq 2 $, which contradicts $r_o(t'^{\perp}_i) > \delta$. 
		 W.l.o.g. we will assume 
		\begin{equation*}
		\begin{split}
		(s_{o} - r , s_{o} + \delta/2) \cap E_{\infty,t'^{\perp}_i} = (s_{o} - r , s_{o} + \delta/2 ). 
		\end{split}
		\end{equation*}
		The other cases are similar. 
		
		Then for every $ a \in (s_o - r, s_o)$ and $a' \in (s_o, s_o+ \delta/2)$, one has that $f_{E_\infty}(t^\perp_i, a, t'^\perp_i -  t^\perp_i, a' - a) = 1$.  Thus 
		\begin{equation*}
		\begin{split}
		\Int(t^{\perp}_{i},t'^{\perp}_{i})  & = 
		\int_{Q_l^i(z_i)} \int_{Q_l^i(z_i)} f_{E_\infty}(t^{\perp}_i,t_i,t'^\perp_{i}-t^{\perp}_i, t'_i-t_i) {\bar K}_{\infty}(t'- t)\dt_{i} \dt'_{i} 
		\\ &\geq \int_{s_o - r}^{s_o} \int_{s_o}^{s_o + \delta/2} f_{E_\infty}(t^{\perp}_i,t_i,t'^\perp_{i}-t^{\perp}_i, t'_i-t_i) {\bar K}_{\infty}(t'- t)\dt'_{i} \dt_{i} 
		\\ &\geq \int_{s_o - r}^{s_o} \int_{s_o}^{s_o + \delta/2}  {\bar K}_{\infty}(t'- t)\dt'_{i} \dt_{i} 
		=\mathbb 1^\infty_{\{(t'^\perp_i,t_i^\perp):\,r^{i}_\lambda(u,t'^{\perp}_{i})\neq r^{i}_{\lambda}(u,t^{\perp}_{i})\}}(t'^\perp_i,t_i^\perp),
		\end{split}
		\end{equation*}
		following at the end the same argument as in case (i).
	\end{enumerate}

\end{proof}

\begin{lemma}
\label{lemma:yukawa_ri_constantNew}
Assume that $E_\infty\subset \R^d$ is a set of locally finite perimeter such that \eqref{eq:gstr381} holds.
Let $\lambda\in(0,l/2)$ and let $r^i_\lambda(u,\cdot)$ be as defined in \eqref{eq:yukawa_defri_local}.  Then, we have that $r^{i}_\lambda(u,\cdot)$ is constant almost everywhere. 
\end{lemma}

\begin{proof}
First of all, since  $\int_{Q_{l}^{\perp}(z_{i}^{\perp})}\mathbb 1^\infty_{B_i}(t_i^\perp)\dt_{i}^\perp<+\infty$, $r_o(t_i^\perp)\geq 1$ for a.e. $t_i^\perp\in Q_l^\perp(z_i^\perp)$.  

Let $B$ be the set defined by
\[
B: = \insieme{ (t'^\perp_{i},t^\perp_{i})\in [0,L)^{d-1}\times[0,L)^{d-1} : r^i_\lambda(u,t_i^\perp)\neq r^i_\lambda(u,t_i'^\perp), \, |t'^\perp_i -t^\perp_i| \leq  \min(\lambda,1/2)}.
\] 
Then, by \eqref{eq:yukawa_gstr11} and Lemma \ref{lemma:yukawa_p2d_stimaSlice}, 
\begin{equation}
\label{eq:yukawa_gstr29}
\begin{split}
 \int_{Q_l^\perp(z_i^\perp)} \int_{Q_l^\perp(z_i^\perp)}\mathbb 1^\infty_B(t_i^\perp, t_i'^\perp)\dt_i\perp\dt_i'^\perp\leq \int_{Q_l^\perp(z_i^\perp)}\int_{Q_l^\perp(z_i^\perp)} \Int(t^{\perp}_{i},t'^{\perp}_{i}) \dt^{\perp}_{i} \dt'^{\perp}_{i}\lesssim l^dN<+\infty.
\end{split}
\end{equation}

Hence, $r^i_\lambda(u, t^{\perp}_{i})=r^i_\lambda(u, t'^{\perp}_{i})$ whenever $| t^{\perp}_{i}- t'^{\perp}_{i}|\leq \min(\lambda,1/2)$ and therefore the statement of the lemma follows.

\end{proof}

From the fact that $r^i_\lambda(u,\cdot)$ is constant almost everywhere for every $u$ and for every $i$, one can deduce that $E_\infty$ must be a checkerboard or a union of stripes.
  We recall that by a checkerboard we mean  any set whose boundary is the union of affine hyperplanes orthogonal to  coordinate axes, and there are at least two of these directions. 

However, the checkerboard can be ruled out immediately.
To see this we consider the contribution to the energy given in a neighbourhood of an edge.  W.l.o.g. we may assume that around this edge the set $E_\infty$ is of the following form $-\varepsilon \leq x_1\leq 0$ and $ -\varepsilon \leq x_2 \leq 0$ and $x_i \in (-\varepsilon, \varepsilon)$ for $i\neq 1,2$. 
Notice that for every $|\zeta|<1$ such that $\zeta_1  + x_1 > 0$, $\zeta_2 + x_2 > 0$ and $\zeta_i\in (-\varepsilon,\varepsilon)$ for $i\neq 1,2$,  the integrand in $\Int(x_1^\perp, (x_1+\zeta_1)^\perp$ is equal to $+\infty$. Therefore also the first term in \eqref{eq:gstr381} must be $+\infty$, which contradicts our assumptions.

	Moreover, since the second term in the \lhs  of \eqref{eq:gstr381} explodes for stripes with minimal width  tending to zero, one has that there exists $\bar{\eta} = \bar{\eta}(N,l)\geq 1$ such that $D_{\bar{\eta}}(E_\infty,Q_l(z))  = 0$.  This contradicts that $D_{\bar{\eta}}(E_{\infty}, Q_l(z)) >\delta$,  which was assumed at the beginning of the proof. 
\end{proof} 


\section{Proof of Theorem \ref{thm:main_yukawa}}
\label{sec:main}

The main purpose of this section is to prove  Theorem~\ref{thm:main_yukawa}.

The general strategy of the proof is similar to the one used to prove Theorem 1.4 in \cite{DR}.  We refer to the outline of the proof of Theorem 1.4 in \cite{DR} for a detailed overview of the ideas of the proof.  Whenever a needed result is already present in \cite{DR}, we will refer to the appropriate lemma or proposition in \cite{DR}.


In order to simplify notation, we will use $A\lesssim B$, whenever there exists a constant $\bar{C}_{d}$ depending only on the dimension $d$ such that $A\leq \bar{C}_d B$.

For notational reasons it is convenient to introduce  the one-dimensional analogue of \eqref{eq:rimdef}. Namely, let $E\subset \R$ be a set of locally finite perimeter and let $s^-, s,s^+\in \partial E$.   We define
\begin{equation}\label{eq:rtau1D}
   \begin{split}
      r_{M}(E,s) := -1 & + \int_\R |\rho| \widehat{K}_{M}(\rho)\d\rho  -  \int_{s^-}^{s} \int_0^{+\infty}  |\chi_{E}(\rho+ u) - \chi_{E}(u)| \widehat{K}_{M} (\rho)\d\rho  \du \\ & - \int_{s}^{s^+} \int_{-\infty}^0  |\chi_{E}(\rho+ u) - \chi_{E}(u)| \widehat{K}_{M} (\rho)\d\rho  \du. 
   \end{split}
\end{equation}

The quantities defined in \eqref{eq:rimdef} and \eqref{eq:rtau1D} are related via $r_{i,M}(E,t^\perp_i,s) = r_{M}(E_{t^\perp_{i}},s)$.



The following is a technical lemma needed in the proof  of Lemma~\ref{lemma:yukawa_stimaLinea}, analogous to Lemma 7.7 in \cite{DR}. It says that given a set $E\subset \R$, and $I\subset \R$ an interval,  then the one-dimensional contribution  to the energy, namely $\sum_{s\in \partial E\cap I}r_{M}(E,s)$, is comparable to the periodic case up to a constant $C_0$ depending only on the dimension.

\begin{lemma}
   \label{lemma:yukawa_1D-opitmization}
   There exists $C_0 > 0$ such that the following holds.
	Let $E\subset \R$  be a set of locally finite perimeter and $I\subset \R$ be an open interval. 
	Let $s^-, s$ and $s^+$ be three consecutive points on the boundary of $E$ and $r_M(E,s)$ defined as in \eqref{eq:rtau1D}.
	Then there exists $M_0>0$ such that for all $M> M_0$ it holds
	\begin{equation}
      \label{eq:yukawa_gstr40}
	\sum_{\substack{s \in \partial E\\ s \in I}} r_{M}(E,s) \geq e^*_{M} |I| - C_0.
	\end{equation}
\end{lemma}

The proof is analogous to that of Lemma 7.7 in \cite{DR} and therefore we omit it.

The next lemma is the so called local stability Lemma.  Informally,  it shows that if  we are in a cube where the set $E\subset\R^d$ is close to a set $E'$ which is a union of stripes in direction $e_i$ (according to Definition \ref{def:yukawa_defDEta}), then it is not convenient to oscillate in direction $e_j$ with $j\neq i$ (namely, on the slices in direction $e_i$ to have points in $\partial E_{t_i^\perp}$).

\begin{lemma}[Local Stability]
   \label{lemma:yukawa_stimaContributoVariazionePiccola}
         Let  $(t^{\perp}_{i}+se_i)\in (\partial E) \cap [0,l)^d$ and $0<\eta_{0}<1$ and $M_0$   as Lemma \ref{lemma:rim}. Then, for every $\eps<\eta_0$ there exists $\tilde M=\tilde M(\tilde{ \eps})>M_0$ such that for every $M > \tilde{M}$  the following holds: assume that 
         \begin{enumerate}[(a)]
            \item $\min(|s-l|, |s|)> \eta_0$
            \item $D^{j}_{\eta}(E,[0,l)^d)\leq\frac {\eps^d} {16 l^d}$ for some $\eta> 0$ and  with $j\neq i$ (this condition expresses that $E\cap [0,l)^d$ is close to stripes oriented along a direction orthogonal to $e_i$)
         \end{enumerate}
         Then $r_{i,M}(E,t^{\perp}_{i},s) + v_{i,M}(E,t^{\perp}_{i},s) \geq 0$. 

\end{lemma}

In the proof one uses a lower bound on the term $v_{i,M}$ which, basing on the Yukawa kernel $\bar K_M$, is different from the one obtained in \cite{DR} and therefore we report it.
\begin{proof}

   Let $s^-,s,s^+$ be three consecutive points for $\partial E_{t_{i}^{\perp}}$. 
   By Lemma \ref{lemma:rim}, for all $0<\eta_{0}<1$, there exists $M_0>0$  such that if $M>M_0$
   \begin{equation*} 
      \begin{split}
         \min(|s-  s^- |, |s^+ -s |) <\eta_{0} \quad\text{then}\quad r_{i,M}(E,t_{i}^{\perp},s) > 0. 
      \end{split}
   \end{equation*} 

   Thus without loss of generality we may assume that $\min(|s - s^- |, |s^+ -s |) \geq \eta_{0}$.

   	Thus, given that, for every $s$, $r_{i,M}(E,t_i^\perp,s)>-e^{-cM}$ for some $c>0$ (see Lemma \ref{lemma:rim}),  one has that 
   	
   	\begin{equation}\label{eq:rmvm}
   	r_{i,M}(E,t^{\perp}_{i},s) + v_{i,M}(E,t^{\perp}_{i},s)\geq -e^{-cM}+ \frac{1}{2d}\int_{s^{-}}^{s^{+}} \int_{\R^{d}} f_{E}(t^{\perp}_{i},u,\zeta^{\perp}_{i},\zeta_{i}) K_{M}(\zeta) \d\zeta\du
   	\end{equation}

Let now $0<\eps<\eta_0$.
By assumption, for some $t_i\in\partial E_{t_i^\perp}$ one of the following holds: 
	\begin{enumerate}[(i)]
	\item $(t_i-\varepsilon,t_i)\subset E_{t_{i}^{\perp}}$ and $(t_i,t_i+ \varepsilon)\subset E_{t_{i}^{\perp}}^c$ 
	\item $(t_i- \varepsilon,t_i)\subset E^c_{t_{i}^{\perp}}$ and $(t_i,t_i+\varepsilon)\subset E_{t_{i}^{\perp}}$ .
\end{enumerate}
 	W.l.o.g., we may assume that (i) above holds and that $i=d$. 

 As shown in \cite[Lemma 6.1, Lemma 7.8]{DR}, hypothesis $(b)$ implies that  
 
    	\begin{equation} 
 \label{eq:yukawa_defBarEps}
 \begin{split}
 \max \Big(\frac{| Q_{\varepsilon}^{\perp}(t^{\perp}_d)\times(t_{d}-\varepsilon,t_{d}) \cap E^{c}|} {| Q_{\varepsilon}^{\perp}(t_d^{\perp})\times(t_{d}-\varepsilon,t_{d})|} ,
 \frac{|Q_{\varepsilon}^{\perp}(t^{\perp}_{d})\times(t_{d},t_{d}+\varepsilon) \cap E|} {| Q^{\perp}_{\varepsilon}(t^{\perp}_{d})\times(t_{d}-\varepsilon,t_{d})|}  
 \Big)  \geq  \frac{7}{16}.
 \end{split}
 \end{equation}

   	Thus, we can further assume that
   	\begin{equation} 
   	\label{eq:yukawa_caso1} 
   	\begin{split}
   	(t_{d}-\varepsilon,t_d) \subset E_{t_d^\perp} \quad \text{and} \quad\frac{| Q_{\varepsilon}^{\perp}(t^{\perp}_{d})\times(t_{d}-\varepsilon,t_{d}) \cap E^{c}|}{| Q^{\perp}_{\varepsilon}(t^{\perp}_{d})\times(t_{d}-\varepsilon,t_{d})|} \geq \frac{7}{16}.
   	\end{split}
   	\end{equation} 

   	For every $s\in (t_{d}-\varepsilon,t_{d})$, $(\zeta_{d}^{\perp} , s) \not \in E$ and $\zeta_{d} +s \in (t_{d},t_{d}+\varepsilon)$ we have that $f_{E}(t^{\perp}_{d},s,\zeta^{\perp}_{d},\zeta_{d}) = 1$. Thus by integrating initially in $\zeta_{d}$ and using \eqref{eq:kmest}, we have that 
   	\begin{equation*} 
   	\begin{split}
   	\int_{t_{d}-\varepsilon}^{t_{d}+\varepsilon}\int_{t_{d}-s}^{t_{d}+\varepsilon-s} \int_{Q^{\perp}_{\varepsilon}(t^{\perp}_{d})} f_{E}(t^{\perp}_{d},s,\zeta^{\perp}_{d},\zeta_{d}) &K_{M}(\zeta)  \d\zeta^\perp_d \d\zeta_d \d s\geq\\
   	&\geq \frac{e^{M(\gamma_M-\eps)}}{\eps^{d-2}} \varepsilon \int_{Q^{\perp}_{\varepsilon}(t^{\perp}_{d})} \int_{t_{d}-\varepsilon}^{t_{d}} |\chi_{E_{t_{d}^{\perp}}}(s) - \chi_{E_{t_{d}^{\perp}+ \zeta^{\perp}_{d}}}( s)  | \ds \d\zeta_{d}^{\perp} \\ 
   	&\geq  \frac{e^{M(\gamma_M-\eps)}}{\eps^{d-2}}  \varepsilon  \int_{Q^{\perp}_{\varepsilon}(t^{\perp}_{d})} \int_{t_{d}-\varepsilon}^{t_{d}} |1 - \chi_{E_{t_{d}^{\perp}+ \zeta^{\perp}_{d}}}( s)  | \ds \d\zeta_{d}^{\perp}  \\
   	&\geq  \frac{e^{M(\gamma_M-\eps)}}{\eps^{d-2}}  \varepsilon | Q^{\perp}_{\varepsilon}(t_{d}^{\perp})\times(t_{d}-\varepsilon, t_{d}) \cap E^c| \geq\frac{7e^{M(\gamma_M-\eps)}\eps^{d+1}}{16\eps^{d-2}},
   	\end{split}
   	\end{equation*} 
   	which tends to $+\infty$ as $M\to+\infty$.
   	
   	Therefore, for $\tilde M$ sufficiently large depending on $\eps$ the r.h.s. of \eqref{eq:rmvm} is positive. Up to a permutation of coordinates, this naturally holds also for $i=2,\ldots,d-1$. Therefore the lemma is proved.

   \end{proof}

   The following Lemma, analogue of Lemma 7.9 in \cite{DR}, gives an estimate from below to the contribution of the energy on a segment of a slice in direction $e_i$.

   \begin{lemma}
      \label{lemma:yukawa_stimaLinea}
      Let $0<\eta_0<1$, $\tilde M$ as in Lemma \ref{lemma:yukawa_stimaContributoVariazionePiccola}. Let $\delta=\eps^d/(16l^d)$ with $0<\eps\leq\eta_0$, $M>\tilde M$ and $l>C_0/(-e^*_M)$, where $C_0$ is the constant appearing in Lemma \ref{lemma:yukawa_1D-opitmization}. Let $t_i^\perp\in[0,L)^{d-1}$ and $\eta>0$.

   The following statements hold: there exists $C_1$ constant independent of $l$ (but depending on the dimension) such that  
      \begin{enumerate}[(i)]
         \item Given $J\subset \R$ such that for every $s\in J$ it holds  $D^{j}_{\eta}(E,Q_{l}(t^{\perp}_{i}+se_i))\leq \delta$ with $j\neq i$, then
            \begin{equation}
               \label{eq:yukawa_gstr20}
               \begin{split}
                  \int_{J} \bar{F}_{i,M}(E,Q_{l}(t^{\perp}_{i}+se_i))\ds \geq - \frac{C_1}{l}.
               \end{split}
            \end{equation}
            Moreover, if $J = [0,L)$, then 
            \begin{equation}
               \label{eq:yukawa_gstr21}
               \begin{split}
                  \int_{J} \bar{F}_{i,M}(E,Q_{l}(t^{\perp}_{i}+se_i))\ds \geq0.
               \end{split}
            \end{equation}
         \item Given $J = (a,b)\subset \R$. 
           If for $s=a$ and $s=b$ it holds $D_\eta^j(E,Q_{l}(t^{\perp}_i+se_i)) \leq \delta$ with $j\neq i$, then 
            \begin{equation}
               \label{eq:yukawa_gstr27}
               \begin{split}
                  \int_{J} \bar{F}_{i,M}(E,Q_{l}(t^{\perp}_{i}+se_i))\ds \geq | J| e^{*}_{M} -\frac{C_1} l,
               \end{split}
            \end{equation}
            otherwise
            \begin{equation}
               \label{eq:yukawa_gstr36}
               \begin{split}
                  \int_{J} \bar{F}_{i,M}(E,Q_{l}(t^{\perp}_{i}+se_i))\ds \geq | J| e^{*}_{M} - C_1l.
               \end{split}
            \end{equation}
            Moreover, if $J = [0,L)$, then
            \begin{equation}
               \label{eq:yukawa_gstr28}
               \begin{split}
                  \int_{J} \bar{F}_{i,M}(E,Q_{l}(t^{\perp}_{i}+se_i))\ds \geq | J| e^{*}_{M}.
               \end{split}
            \end{equation}
   \end{enumerate}
   \end{lemma}

   \begin{proof} For the proof we refer to Lemma 7.9 in \cite{DR}. $C_1$ correspond to $M_0$ there, $M$ to $\tau$, $\eta_0$ to $\tilde \eps$ and $e^*_M$ to $C^*_\tau$. In the proof one uses Lemma \ref{lemma:rim} and Lemma \ref{lemma:yukawa_stimaContributoVariazionePiccola}. Here, the estimate $r_{i,\tau}(E,t_i^\perp,s)\geq -1$ is replaced by $r_{i,M}(E,t_i^\perp,s)\geq -e^{-cM}$ (see Lemma \ref{lemma:rim}).
   	\end{proof}

The purpose of the next lemma is to give a  lower bound on the energy in the case that almost all the volume of $Q_l(z)$ is filled by $E$ or $E^c$.

\begin{lemma}
   \label{lemma:yukawa_stimaQuasiPieno}
   Let $E$ be a set of locally finite perimeter  such that $\min(|Q_{l}(z)\setminus E|, |E\cap Q_{l}(z) |)\leq {\nu} l^d$, for some $\nu>0$. Then 
   \begin{equation*}
      \begin{split}
         \bar F_{M} (E,Q_{l}(z)) \geq -\frac {e^{-cM}\nu d } {\eta_0 },
      \end{split}
   \end{equation*}
   where $\eta_0<1$, provided $M\geq M_0(\eta_0)$ as in Lemma \ref{lemma:rim}.
\end{lemma}

For the proof we refer to Lemma 7.11 in \cite{DR}, substituting the lower bound $r_{i,\tau}(E,t_i^\perp,s)\geq -1$ with $r_{i,M}(E,t_i^\perp,s)\geq -e^{-cM}$ and $\delta$ with $\nu$.

Given the preliminary lemmas, the proof of Theorem \ref{thm:main_yukawa} is analogous to that of \cite{DR}[Theorem 1.4].

	The sets defined in the proof and the main estimates depend on a set of parameters $l,\delta,\rho,N,\eta$ and $M$. The validity of the theorem relies as in \cite{DR} on a suitable choice of such parameters. Since at this point the proof does not present novelties w.r.t. \cite{DR}, we omit it referring to \cite{DR}[Section 7.3] with the following substitutions: $\tau$ is replaced by $M$ and whenever $\tau$ in \cite{DR} has to be chosen smaller than some quantity, $M$ has to be larger than some quantity; $C^*$ is replaced by $e^*$ and $C^*_\tau$ by $e^*_M$. The other parameters remain the same.

%

\section{A nonlocal to local $\Gamma$-limit} 
\label{sec:gammaconv}

As discussed in \cite{BBCH,CCA,IR,GCLW}, one of the possible models used to show gelification in charged colloids and pattern formation is to consider both as attractive and as repulsive term the Yukawa potential, with different signs and appropriate rescaling. 

We therefore consider the following functional: for $E\subset\R^d$, $d\geq3$, $L>0$, $J>0$ and $\beta>1$, let 

\begin{align}
\tilde{\mathcal E}_{\beta,J,L}(E):=&\frac{1}{L^d}\Big(JC_{\beta,L}\int_{[0,L)^d }\int_{ \R^d}|\chi_E(x+\zeta)-\chi_E(x)|K_\beta(\zeta)\d\zeta\dx\notag\\
&-\int_{[0,L)^d }\int_{ \R^d}|\chi_E(x+\zeta)-\chi_E(x)|K_1(\zeta)\d\zeta\dx\Big),\label{E:E}
\end{align}

where $C_{\beta,L}$ is a positive normalization constant defined in \eqref{eq:cbeta} depending on $\beta$ and $L$ and $K_\beta,K_1$ are the Yukawa kernels with parameters $\beta$ and $1$ and with the $1$-norm.

The aim of this section is to prove Theorem \ref{thm:gammayukintro}.

\subsection{The normalization constant}
\label{ss:normconst}
We compute here the normalization constant $C_{\beta,L}$ which allows the first term of \eqref{E:E} to $\Gamma$-converge to the $1$-perimeter as $\beta\to+\infty$.

Let $\bar E\subset\R^d$ $[0,L)^d$-periodic be such that $\bar E\cap[0,L)^d=[L/2,L)\times[0,L)^{d-1}$. Let $H^-:=[0,L/2)\times[0,L)^{d-1}$ and $H^+:=[L/2,L)\times[-L/2,3/2L)^{d-1}$. 

We define $C_{\beta,L}$ as
\begin{equation}\label{eq:cbeta}
C_{\beta,L}:=L^{d-1}\Big(\int_{H^-}\int_{H^+}{|\chi_E(x)-\chi_E(y)|}K_\beta(x-y)\dx\dy\Big)^{-1}
\end{equation}

We give now bounds from above and from below for $C_{\beta,L}$ which are independent of $L$. By definition,

\begin{align}
\int_{H^-}\int_{H^+}&{|\chi_E(x)-\chi_E(y)|}K_\beta(x-y)\dx\dy=\notag\\&=\int_{0}^{L/2}\int_{L/2}^Le^{-\beta|x_1-y_1|}\int_{[0,L)^{d-1}}\int_{[-L/2,3/2L)^{d-1}}\frac{e^{-\beta|x_1^\perp-y_1^\perp|_1}}{(|x_1-y_1|+|x_1^\perp-y_1^\perp|_1)^{d-2}}\dy_1^{\perp}\dx_1^\perp\dy_1\dx_1.\notag
\end{align}
Therefore,
\begin{align}
\int_{H^-}\int_{H^+}&{|\chi_E(x)-\chi_E(y)|}K_\beta(x-y)\dx\dy\leq\notag\\
&\leq L^{d-1}\int_{0}^{L/2}\int_{L/2}^Le^{-\beta(y_1-x_1)}\int_{[0,2L)^{d-1}}\frac{e^{-\beta|\zeta_1^\perp|_1}}{((y_1-x_1)+|\zeta_1^\perp|_1)^{d-2}}\d\zeta_1^\perp\dy_1\dx_1\notag\\
&\lesssim L^{d-1}\int_{0}^{L/2}\int_{L/2}^Le^{-\beta(y_1-x_1)}\int_0^{2L/(y_1-x_1)}\Big(\frac{1}{1+t}\Big)^{d-2}e^{-\beta t(y_1-x_1)}(y_1-x_1)\dt\dy_1\dx_1\notag\\
&\lesssim L^{d-1}\int_{0}^{L/2}\int_{L/2}^Le^{-\beta(y_1-x_1)}\int_0^{2L/(y_1-x_1)}e^{-\beta t(y_1-x_1)}(y_1-x_1)\dt\dy_1\dx_1\notag\\
&\lesssim \frac{L^{d-1}}{\beta^3}(1-e^{-2L\beta})(1-e^{-\beta L/2})^2
\end{align}

On the other hand,

\begin{align}
\int_{H^-}\int_{H^+}&{|\chi_E(x)-\chi_E(y)|}K_\beta(x-y)\dx\dy\geq\notag\\
&\geq L^{d-1}\int_{0}^{L/2}\int_{L/2}^Le^{-\beta(y_1-x_1)}\int_{[0,L/2)^{d-1}}\frac{e^{-\beta|\zeta_1^\perp|_1}}{((y_1-x_1)+|\zeta_1^\perp|_1)^{d-2}}\d\zeta_1^\perp\dy_1\dx_1\notag\\
&\geq L^{d-1}\int_{0}^{L/2}\int_{L/2}^Le^{-\beta(y_1-x_1)}\int_1^{(L)/(2(y_1-x_1))}\Big(\frac{1}{1+t}\Big)^{d-2}e^{-\beta t(y_1-x_1)}(y_1-x_1)\dt\dy_1\dx_1\notag\\
&\geq \frac{L^{d-1}}{\beta^3}\alpha(\beta,L)
\end{align}
where $1\geq\alpha(\beta,L)\geq \bar \alpha>0$ for all $\beta\geq 1$, $L\geq \bar L>0$.

Therefore, $C_{\beta,L}$ satisfies

\begin{equation}\label{eq:cbeta3}
0<\bar c\beta^3\leq C_{\beta,L}\leq\beta^3\bar C<+\infty
\end{equation}

with $\bar c, \bar C$ independent of $L,\beta$ provided $\beta\geq1, L\geq\bar L>0$.

\subsection{$\Gamma$-convergence}
\label{ss:gammaconv}

The main result of this section is the following

\begin{theorem}\label{thm:gammayuk}
	The functionals $\tilde{\mathcal E}_{\beta,J,L}$ defined in \eqref{E:E} $\Gamma$-convergence in the $L^1$ topology as $\beta\to+\infty$ and up to subsequences to the functional $\tilde{\mathcal F}_{J,L}$ defined in \eqref{E:F}.
\end{theorem}

Since the second term in \eqref{E:E} is continuous w.r.t. $L^1$-convergence, in order to prove the theorem it is sufficient to show that
\begin{equation}\label{eq:gamma}
\mathcal P_\beta(\cdot)={C_{\beta,L}}\int_{[0,L)^d }\int_{ \R^d}|\chi_{(\cdot)}(x+\zeta)-\chi_{(\cdot)}(x)|K_\beta(\zeta)\d\zeta\dx\quad\underset{\beta\to+\infty}{\overset{\Gamma}{\longrightarrow}}\quad\per_{1}(\cdot,[0,L)^d).
\end{equation}

W.l.o.g. we consider $d=2$. 

Let $\{E_\beta\}_\beta$ be a sequence of $[0,L)^2$-periodic sets with $ \sup_{\beta}\mathcal P_\beta(E_\beta)<+\infty$.

Let then $\alpha\in(0,1/2)$ and  for all $\beta>1$ define
\[
A_{\alpha,\beta}:=\Big\{t\in[0,L)^2:\,\frac{|E_\beta\cap Q_{1/\beta}(t)|}{|Q_{1/\beta}(t)|}\in(\alpha,1-\alpha)\Big\}.
\]

For every $t\in A_{\alpha,\beta}$
\begin{align}
\mathcal P_\beta(E_\beta)&\geq C_{\beta,L}\int_{Q_{1/\beta}(t)}\int_{Q_{1/\beta}(t)}|\chi_{E_{\beta}}(x)-\chi_{E_{\beta}}(y)|K_{\beta}(x-y)\dx\dy\geq c\frac{\alpha(1-\alpha)}{\beta},\label{eq:stimaalpha1d}
\end{align}
since $C_{\beta,L}$ goes like $\beta^3$ (see \eqref{eq:cbeta3}) and since when $|x-y|\leq1/\beta$ the kernel $K_{\beta}$ is bounded from below by a constant.

Let now $N(\beta,\alpha)\in\N$ be the maximal number of disjoint cubes $Q_{1/\beta}(t_i)$ centred in $t_i\in A_{\alpha,\beta}$ of side length $1/\beta$.

One has that 
\[
\mathcal P_\beta(E_\beta)\geq cN(\alpha,\beta)\frac{\alpha(1-\alpha)}{\beta}
\] 
and from the uniform upper bound on $\mathcal P_\beta(E_\beta)$ 
\[
N(\alpha,\beta)\leq c(\alpha)\beta.
\]

For sure $A_{\alpha,\beta}\subset\cup_{i=1}^{N(\alpha,\beta)}Q_{4/\beta}(t_i)$, from which it follows that

\begin{equation}
|A_{\alpha,\beta}|\leq c(\alpha)\beta \frac{16}{\beta^2}\quad\longrightarrow\quad0\quad \text{ as $\beta\to+\infty$.}
\end{equation}

Let $g$ be any weak*-$L^\infty$ limit of subsequences of $\chi_{E_{\beta}}$, as $\beta\to+\infty$. Then for any $t\in \R$ one has that $g\in[0,1]$.  Moreover, from the above reasoning, for every $\alpha\in(0,1/2)$ there exists a null set $X_\alpha$ such that, for all $t\in[0,L)^2\setminus X_\alpha$ either  $g(t)\geq1-\alpha$ or $g(t)\leq\alpha$. From this it follows that $g=\chi_{E}$ for some $[0,L)^2$-periodic set $E$. 

In particular, the weak*-$L^\infty$ convergence of $E_{\beta}$ to $E$ can be upgraded to strong $L^1$ convergence.

We claim that $E$ is of finite perimeter in $[0,L)^2$. Indeed, consider the set
\[
E^{1/2}:=\Big\{t\in\R^2:\,\exists\,\lim_{r\to0}\fint_{B_{r}(t)}\chi_{E}(u)\du=1/2\Big\}.
\]   
By Federer's characterization of the sets of finite perimeter,  one has that $\hausd^1(E^{1/2}\cap[0,L)^2) < +\infty$ if and only if the set $E$ is of finite perimeter. 

Let us consider a fine covering of $E^{1/2}\cap[0,L)^2$ with cubes $\{Q_{r(t)}(t)\}_{t\in\mathcal T}$ such that 
\begin{equation}
\frac{|E\cap Q_{r(t)}( t)|}{|Q_{r(t)}( t)|}\in\Big(\frac12-\eps,\frac12+\eps\Big).
\end{equation}

Thanks to the covering Theorem of Besicovitch, there exist $N=N(d)$ collections of disjoint cubes $\{Q_{r(t_i^j)}(t_i^j)\}_{i\in\mathcal T_j}\subset\{Q_{r(t)}(t)\}_{t\in\mathcal T}$, $j=1,\dots,N$ such that
\[
E^{1/2}\subset \bigcup_{j=1,\dots,N}\bigcup_{i\in\mathcal T_j}Q_{r(t_i^j)}(t_i^j).
\]

As a consequence,
\begin{equation}\label{eq:stimah}
\mathcal H^1(E^{1/2})\leq \sum_{j=1}^{N}\sum_{i \in\mathcal T_j}\sqrt2r(t_i^j)
\end{equation}

In order to prove that the r.h.s. of \eqref{eq:stimah} is bounded, we claim the following: for $\beta$ large enough, the number $T_i^j$ of disjoint cubes of side length $1/\beta$, $\{Q_{1/\beta}(t_m)\}_{m=1}^{T^j_i}$ contained in $Q_{r(t_i^j)}(t_i^j)$ for fixed $i,j$ such that $\frac{|E\cap Q_{1/\beta}( t_m)|}{|Q_{1/\beta}( t_m)|}\in\Big[\frac12-\eps,\frac12+\eps\Big]$ is bigger or equal than $cr(t_{i}^j)\beta$ for some constant $c>0$.

Before proving the claim, let us see how this gives an upper bound for \eqref{eq:stimah}. 

Since the sets $E_\beta$ converge in $L^1$ to $E$, then for $\beta$ sufficiently large and independent of $i,j$ 
\[
\frac{|E_\beta\cap Q_{1/\beta}( t_m)|}{|Q_{1/\beta}( t_m)|}\in\Big(\frac12-2\eps,\frac12+2\eps\Big).
\]
Therefore, 
\[
\mathcal P_\beta(E_\beta)\geq \sum_{i\in\mathcal T_j} \frac{\bar c}{4\beta^4}C_{\beta,L}T_i^j\geq\sum_{i\in\mathcal T_j}\frac{\bar c}{4\beta^4}C_{\beta,L} cr(t_{i}^j)\beta\geq \sum_{i\in\mathcal T_j}\tilde cr(t_{i}^j),
\]

from which by the upper bound on $\mathcal P_\beta(E_\beta)$ the finiteness of \eqref{eq:stimah} follows.

We now prove the lower bound on $T_i^j$ contained in the claim.
Define the following sets:
\begin{align*}
A_-&:=\Big\{x\in Q_{r(t_i^j)}(t_i^j):\,\frac{|E\cap Q_{1/\beta}(x)|}{|Q_{1/\beta}(x)|}<\frac12-\eps\Big\},\\
A_+&:=\Big\{x\in Q_{r(t_i^j)}(t_i^j):\,\frac{|E\cap Q_{1/\beta}(x)|}{|Q_{1/\beta}(x)|}>\frac12+\eps\Big\},\\
A&:=\Big\{x\in Q_{r(t_i^j)}(t_i^j):\,\frac{|E\cap Q_{1/\beta}(x)|}{|Q_{1/\beta}(x)|}\in\Big[\frac12-\eps,\frac12+\eps\Big]\Big\}.
\end{align*}
The set $A$ separates $A_-$ and $A_+$, meaning that for every segment $[x,y]$  connecting $x\in A_-$ with $y\in A_+$ there exists $z\in [x,y]$ with $z\in A$. Therefore if we show that 
\begin{equation}\label{eq:perbound}
\mathrm{Per}(\partial A_{\pm})\geq cr(t_i^j),
\end{equation} the claim is proved. 
The lower bound \eqref{eq:perbound} is a consequence of the isoperimetric inequality applied to $A_-$ or $A_+$. Indeed, the measure of each of $A_\pm$ is bigger or equal than $r(t_i^j)^2/4$ for $\beta$ sufficiently large depending only on $E\cap [0,L)^2$, being almost all points in  the sets $E\cap Q_{r(t_i^j)}$ and $E^c\cap Q_{r(t_i^j)}$ respectively of density $1$ and $0$ and of total measure bigger or equal than $(1/2-\eps)r(t_i^j)^2$.

Let us call $\mathcal P$ a $\Gamma$-limit, up to subsequences, of the l.h.s. of \eqref{eq:gamma}.

The next step to prove Theorem \ref{thm:gammayuk} is the following

\begin{lemma}\label{lemma:gammayuk1}
	\begin{equation}\label{eq:phigamma}
	\mathcal P(E)=\int_{\partial E\cap[0,L)^d}\phi(\nu_E(x))\,\d\Hcal^{d-1}(x),
	\end{equation}	
	where $\partial E$ is the reduced boundary of $E$, $\phi(\nu)=\underset{\eps\to0}{\lim}\frac{\mathcal P(E_\nu\cap[0,\eps)^d)}{\eps^{d-1}}$ and $E_\nu=\{x\cdot\nu\leq0\}$.
\end{lemma}

Thanks to the results obtained in \cite{Fon}, if one shows that, for all sets $E$ with $\partial E\in C^2$
\begin{equation}\label{eq:stima}
\frac{1}{C}(1+\mathcal H^{d-1}(E))\leq \mathcal P(E)\leq C(1+\mathcal H^{d-1}(E))
\end{equation}
then the representation Lemma \ref{lemma:gammayuk1} holds.

Let us then prove \eqref{eq:stima}. 
First of all, consider the upper bound.  Being the boundary of $E\cap[0,L)^d$ compact and of class $C^2$, there exists a covering of it with finitely many  cubes $\{C(x_j, \nu_j,h_j)\}_{j=1}^N$, where $x_j\in\R^d$, $\nu_j\in\SS^{d-1}$, $h_j>0$, of the form
\begin{equation}
C(x_j,\nu_j, h_j)=\{y\in\R^d:\,y_j:=(y-x_j)\cdot\nu_j\leq h_j,\quad |y-y_j\nu_j-x_j|_\infty\leq h_j \}
\end{equation}
and such that in each of the cubes $4\bigl(C(x_j,\nu_j, h_j)-x_j\bigr)$ the boundary of $E-x_j$ is the graph of a Lipschitz function of the plane $\{y\cdot\nu_j=0\}$, with Lipschitz constant bounded by some constant independent of $j$. 

Let $C_j:=C(x_j,\nu_j, h_j)$, $E_j:=E\cap C_j$.

Then,
\begin{align}
{C_{\beta,L}}\int_{[0,L)^d }\int_{ \R^d}&|\chi_{E}(x+\zeta)-\chi_{E}(x)|\frac{e^{-\beta|\zeta|_1}}{|\zeta|_1^{d-2}}\d\zeta\dx\leq\notag\\
&\leq C+\sum_{j=1}^NC_{\beta,L}\int_{C_j}\int_{C_{j-1}\cup C_j\cup C_{j+1}-x}|\chi_{E}(x+\zeta)-\chi_{E}(x)|\frac{e^{-\beta|\zeta|_1}}{|\zeta|_1^{d-2}}\d\zeta\dx
\end{align}
where the sets $C_0$ and $C_{N+1}$ are defined by periodicity and $C$ is independent of $\beta$. This is due to the decay of the kernel, which gives finite weight to interactions at distances bigger than $\min_jh_j$.

Now, let us consider one of the contributions above where we assume for simplicity that $\nu_j=e_1$, $x_j=0$, $\eps=\max\{h_{j-1},h_j,h_{j+1}\}$ and we denote by $\Phi$ the Lipschitz map that maps $H^-_\eps=[-\eps/2,0]\times[-\eps/2,\eps/2]^{d-1}$ in $E_j$ and $H^+_\eps=[0,\eps/2]\times[-3/2\eps,3/2\eps]^{d-1}$ in $ (2C_j)\setminus(E_{j}\cup E_{j-1}\cup E_{j+1})$:

\begin{align}
{C_{\beta,L}}\int_{C_j}&\int_{(C_{j-1}\cup C_j\cup C_{j+1})-x}|\chi_{E}(x+\zeta)-\chi_{E}(x)|\frac{e^{-\beta|\zeta|_1}}{|\zeta|_1^{d-2}}\d\zeta\dx\sim\notag\\
&\sim{C_{\beta,L}}\int_{H^-_\eps}\int_{H^+_\eps}|\chi_E(\Phi(x))-\chi_E(\Phi(y))|\frac{e^{-\beta|\Phi(x)-\Phi(y)|_1}}{|\Phi(x)-\Phi(y)|_1^{d-2}}|D\Phi|(x)|D\Phi|(y)\dx\dy.
\end{align}
Carrying on analogous calculations to those of  Section \ref{ss:normconst} one obtains that the above limit as $\beta\to+\infty$ is less or equal than a constant times $\eps^{d-1}$, namely comparable to the measure of $\partial (E\cap C_j)$. Therefore, the estimate from above is proved.

Now, let us prove the estimate from below  in \eqref{eq:stima}.
To this aim, take a Besicovitch covering of $\partial E$ with cylinders $\{C^\alpha_j\}_{\underset{\alpha=1,\dots,N_0}{j=1,\dots,N_\alpha}}$ such that for all $\alpha\in\{1,\dots,N_0\}$ the sets $\{C^\alpha_j\}_{j=1,\dots,N_\alpha}$ are disjoint and such that in each of the cubes  $4\bigl(C_j^\alpha-x_j^\alpha\bigr)$ the boundary of $E-x_j^\alpha$ is the graph of a Lipschitz function of the plane $\{y\cdot\nu_j^\alpha=0\}$ with Lipschitz constant uniformly bounded in $j$.

Then
\begin{align}
{C_{\beta,L}}\int_{[0,L)^d }\int_{ \R^d}&|\chi_{E}(x+\zeta)-\chi_{E}(x)|\frac{e^{-\beta|\zeta|_1}}{|\zeta|_1^{d-2}}\d\zeta\dx\geq\notag\\
&\geq -C+\sum_{j=1}^{N_\alpha}C_{\beta,L}\int_{C_j^\alpha}\int_{C^\alpha_{j-1}\cup C_j^\alpha\cup C^\alpha_{j+1}-x}|\chi_{E}(x+\zeta)-\chi_{E}(x)|\frac{e^{-\beta|\zeta|_1}}{|\zeta|_1^{d-2}}\d\zeta\dx
\end{align}
where the sets $C_0^\alpha$ and $C_{N+1}^\alpha$ are defined by periodicity and $C$ is independent of $\beta$. This is due again to the decay of the kernel and the fact that the sets $\{C_j^\alpha\}_{j=1,\dots,N_\alpha}$ are disjoint. After making a change of variables with the map $\Phi^\alpha$ that maps part of the sets $E\cap C^\alpha_j$, $E\cap C^\alpha_{j-1}$, $E\cap C^\alpha_{j+1}$ into adjacent half squares of side length $\min\{h_j,h_{j-1},h_{j+1}\}$
The single contributions of the sets $C_j^\alpha$ can be estimated in the same way as in the estimates from below in Section \ref{ss:normconst}, leading to something of the order of $\partial (E\cap C_j^\alpha)$. Applying the same reasoning to the other families $\{C_j^\beta\}$ with $\beta\neq\alpha$, $\beta\in\{1,\dots,N_0\}$ one obtains the estimate from below as well.

The second step to prove Theorem \ref{thm:gammayuk} after Lemma \ref{lemma:gammayuk1} is to characterize the function $\phi$ in \eqref{eq:phigamma}.

W.l.o.g. we consider $d=2$. Then the kernel is given by $K_{\beta}(\zeta):=-{e^{-\beta|\zeta|_1}}\ln(|\zeta|_1)$. 

Let us recall 
\begin{equation*}
   \begin{split}
|\chi_E(x)-\chi_E(x+\zeta)|= &|\chi_E(x)-\chi_E(x+\zeta_1)| +|\chi_E(x+\zeta_1)-\chi_E(x+\zeta)| \\ & -2|\chi_E(x)-\chi_E(x+\zeta_1)||\chi_E(x+\zeta_1)-\chi_E(x+\zeta)|.
   \end{split}
\end{equation*}

Integrating and using the $[0,L)^2$-periodicity of $E$
\begin{align}
\int_{[0,L)^2 }\int_{ \R^{2}}|\chi_E(x)&-\chi_E(x+\zeta)|K_{\beta}(\zeta)\d\zeta\dx=\int_{[0,L)^2 }\int_{ \R^{2}}|\chi_E(x)-\chi_E(x+\zeta_1)|K_{\beta}(\zeta)\d\zeta\dx\notag\\&+\int_{[0,L)^2 }\int_{ \R^{2}}|\chi_E(x)-\chi_E(x+\zeta_2)|K_{\beta}(\zeta)\d\zeta\dx\notag\\
&-2\int_{[0,L)^2 }\int_{ \R^{2}}|\chi_E(x)-\chi_E(x+\zeta_1)||\chi_E(x)-\chi_E(x+\zeta_2)|K_{\beta}(\zeta)\d\zeta\dx\label{eq:splitting}
\end{align} 

Let now $E$ be given, up to translations, of the form $\{x\cdot\nu\leq 0\}\cap[-2\eps,2\eps]^{2}$. 

Then,

\begin{align}
C_{\beta,L}\int_{-\eps}^{\eps}\int_{-\eps}^\eps\int_{\R}\int_{ \R}|\chi_E(x)-&\chi_E(x+\zeta_1)|K_{\beta}(\zeta_1,\zeta_2)\d\zeta_2\d\zeta_1\dx_1\dx_2=\notag\\
&C_{\beta,L}\int_{-\eps}^{\eps}\int_{-\eps}^\eps\int_{\R}|\chi_{E_{x_2}}(x_1)-\chi_{E_{x_2}}(x_1+\zeta_1)|\widehat K_{\beta}(\zeta_1)\d\zeta_1\dx_1\dx_2
\end{align}

which converges, as  $\beta\to+\infty$, to
\begin{equation}\label{eq:6.14}
\int_{-\eps}^{\eps}\int_{ \partial E_{x_2}}\d\Hcal^0(x_1)\dx_2=\per_{11}(E,[\eps,\eps)^2).
\end{equation}

Analogously, the second term in \eqref{eq:splitting} converges to

\begin{equation}\label{eq:6.15}
\int_{-\eps}^{\eps}\int_{ \partial E_{x_1}}\d\Hcal^0(x_2)\dx_1=\per_{12}(E,[\eps,\eps)^2).
\end{equation}

We claim that the third term in \eqref{eq:splitting} is of lower order and therefore converges to $0$ as $\beta\to+\infty$.

We have that
\begin{align}
C_{\beta,L}&\int_{[-\eps,\eps)^2 }\int_{ \R^{2}}|\chi_E(x)-\chi_E(x+\zeta_1)||\chi_E(x)-\chi_E(x+\zeta_2)|K_{\beta}(\zeta)\d\zeta\dx\sim\notag\\
&\sim C_{\beta,L}\int_0^\eps\int_{-\eps}^{x_2\tan\theta}\int_{-x_1+x_2\tan\theta}^{+\infty}\int_{-x_2+\frac{x_1}{\tan\theta}}e^{-\beta|\zeta_{1}|}e^{-\beta|\zeta_2|}(-\ln(|\zeta_1|+|\zeta_2|))\d\zeta_2\d\zeta_1\dx_1\dx_2,
\label{eq:6.16}
\end{align}
where $\zeta_1+\zeta_2\geq-x_2+\frac{x_1}{\tan\theta}-x_1+x_2\tan\theta$, $\theta$ is the minus the angle between $e_1$ and $\nu$ and, w.l.o.g. is assumed to be between $0$ and $-\pi/4$. 

Since now in comparison to Section \ref{ss:normconst} the variable $x_2$ appears as well in the independent variable of integration, simple estimates show that such term goes like $\frac{C_{\beta,L}}{\beta^4}$ and therefore vanishes as $\beta\to+\infty$.

\end{document}